\documentclass[a4paper,11pt]{article}
\usepackage{latexsym,amssymb,amsmath,amsthm,amsxtra,xypic}
\usepackage{stmaryrd}
\usepackage{amsfonts}
\usepackage{bbm}
\usepackage{amscd}
\usepackage{mathrsfs}
\usepackage[dvips]{graphicx}
\usepackage[all]{xy}

\newcommand{\allowpagebreak}
\allowdisplaybreaks

\newtheorem{thm}{Theorem}[section]
\newtheorem{lem}[thm]{Lemma}
\newtheorem{cor}[thm]{Corollary}
\newtheorem{prop}[thm]{Proposition}

\theoremstyle{definition}
\newtheorem{defi}[thm]{Definition}
\newtheorem{rmk}[thm]{Remark}

\newcommand {\emptycomment}[1]{}

\def\Z{\mathbb{Z}}

\def\C{\mathbb{C}}

\newcommand{\al}{\alpha}
\newcommand{\be}{\beta}

\newcommand{\si}{\sigma}

\newcommand{\g}{\mathfrak g}
\newcommand{\hh}{\mathfrak h}

\newcommand{\hg}{\widehat{\mathfrak g}}
\newcommand{\hp}{\widehat{\mathfrak p}}

\newcommand{\h}{{V}}
\newcommand{\frkg}{\mathfrak g}
\newcommand{\fg}{\mathfrak g}
\newcommand{\frkh}{{V}}

\newcommand{\gh}{\mathfrak g \oplus {V}}

\newcommand{\pp}{\mathfrak p}

\newcommand{\dM}{\mathrm{d}}
\newcommand{\E}{\mathrm{E}}

\newcommand{\Hom}{\mathrm{Hom}}
\newcommand{\Der}{\mathrm{Der}}

\newcommand{\Aut}{\mathrm{Aut}}

\newcommand{\End}{\mathrm{End}}

\newcommand{\ad}{\mathrm{ad}}

\newcommand{\id}{\mathrm{id}}

\newcommand{\brh}[1]{   [    #1   ]_{\gh}   }

\allowdisplaybreaks

\begin{document}

\title{ {Wells exact sequence for  automorphisms and derivations of Leibniz 2-algebras
\author{Wei Zhong, Tao Zhang}
 }}

\date{}
\maketitle



\begin{abstract}
In this paper, we investigate the inducibility of pairs of automorphisms and derivations in Leibniz 2-algebras. To begin, we provide essential background information on Leibniz 2-algebras and its cohomology theory. Next, we examine the inducibility of pairs of automorphisms and derivations, with a focus on the analog of Wells exact sequences in the context of Leibniz 2-algebras. We then analyze the analogue of Wells short exact sequences as they relate to automorphisms and derivations within this framework. Finally, we investigate the special case of crossed modules over Leibniz algebras.
\end{abstract}


\section{Introduction}
Recently, there has been increased attention to higher categorical structures in both mathematics and physics. One approach to constructing these structures is through the process of categorification, which involves elevating existing mathematical concepts to higher dimensions. A fundamental example of this is the 2-vector space, which represents the categorification of a traditional vector space. Additionally, Loday \cite{JL} introduced Leibniz algebras, which serve as a generalization of Lie algebras. See \cite{AO,Bra,Loday} for some interesting results about Leibniz algebras. By incorporating a compatible Leibniz algebra structure into a 2-vector space, we can derive a Leibniz 2-algebra \cite{SL}.

In this paper, we aim to explore the inducibility of pairs of automorphisms and derivations in Leibniz 2-algebras, as this concept is closely related to algebraic extensions. The extensibility problem was first addressed by Wells in the context of abstract groups, see \cite{CW}. This theme was later expanded to include various algebraic structures, such as Lie algebras, Rota-Baxter Lie algebras, Rota-Baxter groups, associative conformal algebras, and Lie-Yamaguti algebras (see \cite{VM,AN,SS,BJ,PH,SA,SAS,IM,CW} and the references therein). Building on these findings, we investigate the inducibility of pairs of automorphisms in abelian extensions of Leibniz 2-algebras. This inquiry results in the derivation of an analogue of the Wells exact sequences. Furthermore, we analyze the inducibility of pairs of derivations and establish the corresponding Wells exact sequences for these abelian extensions. Additionally, we also conduct a study on the derivations and automorphisms of crossed modules over Leibniz algebras.

The paper is organized as follows. In Section 2, we recall some basic information about the Leibniz 2-algebras Lie algebra and its cohomology theory. In Section 3, we examine the inducibility problem of a pair of derivations related to an abelian extension of Leibniz 2-algebras and establish Wells exact sequences for these derivations within this framework. In Section 4, we study the inducibility problem of a pair of automorphisms concerning an abelian extension of Leibniz 2-algebras. Additionally, we explore the analogue of Wells short exact sequences related to automorphisms in the context of Leibniz 2-algebras. In Section 5, we focus on the special case of crossed modules over Leibniz algebras.

\section{Abelian extension of Leibniz 2-algebras}\label{sec:cohom}
In this section, we  revisit some notations about Leibniz algebras and Leibniz 2-algebras.
Recall that a Leibniz algebra is a vector space $\pp$  endowed with a linear map $[\cdot,\cdot]:\pp\otimes\pp\longrightarrow\pp$
satisfying
$$
[x,[y,z] ] =[[x,y] ,z] +[y,[x,z]],
$$
for all $x,y,z\in \pp.$

For a Leibniz algebra ${\pp}$, a representation of  ${\pp}$ is a vector space ${M}$ together with two bilinear maps
$$\cdot: {\pp}\otimes {M}\to {M} \,\,\, \text{and}\,\,\, \cdot: {M}\otimes {\pp}\to {M},$$
satisfying the following three axioms
\begin{itemize}
\item[$\bullet$] {\rm(LLM)}\quad  $[x,y]\cdot m=x\cdot (y\cdot  m)-y\cdot (x\cdot  m)$,
\item[$\bullet$] {\rm(LML)}\quad  $m\cdot [x,y]=(m\cdot  x)\cdot  y+x\cdot (m\cdot  y)$,
\item[$\bullet$] {\rm(MLL)}\quad  $m\cdot [x,y]=x\cdot (m\cdot  y)-(x\cdot m)\cdot y.$
\end{itemize}
By (LML) and (MLL) we also have
\begin{itemize}
\item[$\bullet$] {\rm(MMM)}\quad  $(m\cdot x)\cdot y+(x\cdot m)\cdot y=0.$
\end{itemize}

The Loday-Pirashvili cohomology of $\pp$ with coefficients in $V$ is the cohomology of the cochain complex
$C^n(\pp,V)=\Hom_R(\otimes^n\pp,V)$ with the coboundary operator $d_L:C^n(\pp,V)\longrightarrow
C^{n+1}(\pp,V)$ defined by
\begin{eqnarray}
\nonumber&&d_L f^n(x_1,\dots,x_{n+1})\\
\nonumber&=&\sum_{i=1}^n(-1)^{i+1}x_i\cdot f^n(x_1,\dots,\widehat{x_i},\dots,x_{n+1})+(-1)^{n+1}f^n(x_1,\dots,x_k)\cdot x_{n+1}\\
\label{formulapartial}&&+\sum_{1\leq i<j\leq n+1}(-1)^i f^n(x_1,\dots,\widehat{x_i},\dots,x_{j-1},[x_i,x_j],x_{j+1},\dots,x_{n+1}).
\end{eqnarray}
The fact that $d_L\circ d_L=0$ is proved in \cite{LP}.

\begin{defi}[\cite{SL}]\label{defi:Leibniz2}
  A Leibniz 2-algebra $\g=\g_{{1}}\oplus \g_{0}$ consists of the following data:
\begin{itemize}
\item[$\bullet$] a complex of vector spaces $\dM:\g_{{1}}{\longrightarrow}\g_0,$ 

\item[$\bullet$] bilinear maps $[\cdot,\cdot]:\g_i\times \g_j\longrightarrow
\g_{i+j}$, where $0\leq i+j\leq 1$,

\item[$\bullet$] a  trilinear map $l_3:\g_0\times \g_0\times \g_0\longrightarrow
\g_{{1}}$,
\end{itemize}
   such that for any $x,y,z,t\in \g_0$ and $a,b\in \g_{{1}}$, the following equalities are satisfied:
\begin{itemize}
\item[$\rm(a)$] $\dM [x,a]=[x,\dM a],$
\item[$\rm(b)$]$\dM [a,x]=[\dM a,x],$
\item[$\rm(c)$]$[\dM a,b]=[a,\dM b],$
\item[$\rm(d)$]$\dM l_3(x,y,z)=[x,[y,z]]-[[x,y],z]-[y,[x,z]],$
\item[$\rm(e)$]$ l_3(x,y,\dM a)=[x,[y,a]]-[[x,y],a]-[y,[x,a]],$
\item[$\rm(f)$]$ l_3(x,\dM a,y)=[x,[a,y]]-[[x,a],y]-[a,[x,y]],$
\item[$\rm(g)$]$ l_3(\dM a,x,y)=[a,[x,y]]-[[a,x],y]-[x,[a,y]],$
\item[$\rm(h)$] the Jacobiator identity:
\begin{eqnarray*}
&&[x,l_3(y,z,t)]-[y,l_3(x,z,t)]+[z,l_3(x,y,t)]+[l_3(x,y,z),t]-l_3([x,y],z,t)-l_3(y,[x,z],t)\\
&&-l_3(y,z,[x,t])+l_3(x,[y,z],t)+l_3(x,z,[y,t])-l_3(x,y,[z,t])=0.
\end{eqnarray*}
\end{itemize}
\end{defi}
We usually denote a Leibniz 2-algebra by $(\frkg;\dM_\g,[\cdot,\cdot]_\g,l_3^\g)$, or simply by $\g$.

When $l_3=0$, from Condition $(d)$ in Definition \ref{defi:Leibniz2}, we get that $\fg_0$ is a Leibniz algebra, from $(e)$, $(f)$ and $(g)$ in Definition \ref{defi:Leibniz2}, we obtain that  $\fg_1$ is a representation of  Leibniz algebra $\fg_0$.
Thus a Leibniz 2-algebra is also called a homotopy Leibniz algebra.
When $\dM=0$, then $\fg_0$ a Leibniz algebra, $\fg_1$ a representation of  $\fg_0$ and $l_3\in \Hom(\otimes^3 \fg_0,\fg_1)$ a Loday-Pirashvili 3-cocycle.
This is called a skeletal Leibniz 2-algebra.
In general, $\dM\neq 0$ and $l_3\neq  0$, a Leibniz 2-algebra is a two-term truncation of the so-called $\fg_\infty$ algebra or homotopy  Leibniz algebra.

\begin{defi}\label{defi:Lie-2hom}
Let $(\frkg;\dM_\g,[\cdot,\cdot]_\g,l_3^\g)$ and $(\frkg';\dM',[\cdot,\cdot]',l_3')$ be Leibniz 2-algebras. A
Leibniz 2-algebra homomorphism $F$ from $\frkg$ to $ \frkg'$ consists of  linear maps $F_0:\frkg_0\rightarrow \frkg_0',~F_1:\frkg_{{1}}\rightarrow \frkg_{{1}}'$ and $F_{2}: \frkg_{0}\times \frkg_0\rightarrow \frkg_{{1}}'$,
such that the following equalities hold for all $ x,y,z\in \frkg_{0},
a\in \frkg_{{1}},$
\begin{itemize}
\item [$\rm(i)$] $F_0\dM_\g=\dM'F_1$,
\item[$\rm(j)$] $F_{0}[x,y]_\g-[F_{0}(x),F_{0}(y)]'=\dM'F_{2}(x,y),$
\item[$\rm(k)$] $F_{1}[x,a]_\g-[F_{0}(x),F_{1}(a)]'=F_{2}(x,\dM_\g (a))$,
\item[$\rm(l)$] $F_{1}[a,x]_\g-[F_{1}(a),F_{0}(x)]'=F_{2}(\dM_\g (a),x)$,
\item[$\rm(m)$] $F_1(l_3^\g(x,y,z))-l_3'(F_0(x),F_0(y),F_0(z))=F_2(x, [y,z])- F_2([x,y],z) - F_2(y,[x,z])$\\
  $+ [F_0(x), F_2(y,z)]' - [F_2(x,y), F_0(z)]' - [F_0(y), F_2(x,z)]'.$
\end{itemize}
 If $F_2=0$, the homomorphism $F$ is called a strict homomorphism.
\end{defi}

Let $F:\frkg\to \frkg'$ and $G:\frkg'\to \frkg''$ be two homomorphisms, then their composition $GF:\frkg\to \frkg''$
is a homomorphism defined as $(GF)_0=G_0\circ F_0:\frkg_0\to \frkg''_0$, $(GF)_1=G_1\circ F_1:\frkg_{{1}}\to \frkg''_{{1}}$
and
$$(GF)_2=G_2\circ (F_0\times F_0)+G_1\circ F_2:\frkg_0\times \frkg_0\to \frkg''_{{1}}.$$

Let $ {V}:V_{{1}}\stackrel{\partial}{\longrightarrow} V_0$ be a 2-term complex of vector spaces.
Then we get a new 2-term complex of vector spaces $\End({V}):\End^1(\mathbb
V)\stackrel{\delta}{\longrightarrow} \End^0_\partial({V})$ by
defining $\delta(A)=\partial\circ A+A\circ\partial$ for all
$A\in\End^1({V})$, where $\End^1({V})=\End(V_0,V_{{1}})$
and
$$\End^0_\partial({V})=\{X=(X_0,X_1)\in \End(V_0,V_0)\oplus \End(V_{{1}},V_{{1}})|~X_0\circ \partial=\partial\circ X_1\}.$$
Define  $l_2:\wedge^2 \End({V})\longrightarrow \End(\mathbb
V)$ by setting:
\begin{equation}
\left\{\begin{array}{l}l_2(X,Y)=[X,Y]_C,\\
l_2(X,A)=[X,A]_C,\\
l_2(A,A')=0,\end{array}\right.\nonumber
\end{equation}
 for all $X,Y\in
\End^0_\partial({V})$ and $A,A'\in \End^1({V}),$ where
$[\cdot,\cdot]_C$ is the graded commutator.

\begin{thm} [\cite{LadaMarkl}] \label{thm:End(V)}With the above notations,
$(\End({V}),\delta,l_2)$ is a strict Lie $2$-algebra.
\end{thm}

\begin{defi}\label{Def:derivation}
Let $(\g:\g_{1}\stackrel{\dM_\frkg}{\longrightarrow}
\g_{0},[\cdot,\cdot]_\frkg,l_3^\g)$ be a Leibniz 2-algebra. A derivation
of degree $0$
 of $\frkg$ consists of
\begin{itemize}
  \item[$\bullet$] ${D}_0: \g_0\longrightarrow\g_0, {D}_1:\g_1\longrightarrow\g_{1}$ such that ${D}_0 \circ d = d \circ {D}_1$,
  \item[$\bullet$] a bilinear map $D_2:\g_{0}\times\g_0\longrightarrow\g_{1},$
\end{itemize}
such that for all $x,y,z\in\g_0$ and $a\in\g_1$
\begin{itemize}
\item [$\rm(n)$] $D_0[x,y]_\frkg-[D_0x,y]_\frkg-[x,D_0y]_\frkg=\dM_{\g}D_2(x,y),$
\item[$\rm(o)$] $D_0[x,a]_\frkg-[D_0x,a]_\frkg-[x,D_1a]_\frkg=D_2(x,\dM_\g a),$
\item[$\rm(p)$] $D_0[a,x]_\frkg-[D_1a,x]_\frkg-[a,D_0x]_\frkg=D_2(\dM_\g a, x),$
\item[$\rm(q)$] $D_0l_{3}^\g(x,y,z)-l_{3}^\g(D_0x,y,z)-l_{3}^\g(x,D_0y,z)-l_{3}^\g(x,y,D_0z)=D_2([x,y]_\frkg,z)$\\
    $+D_2(y,[x,z]_\frkg)-D_2(x,[y,z]_\frkg)+[D_2(x,y),z]_\frkg+[y,D_2(x,z)]_\frkg-[x,D_2(y,z)]_\frkg.$
\end{itemize}
\end{defi}

\begin{defi}[\cite{Zhang}]
A representation of a Leibniz 2-algebra $(\frkg; \dM_\g,[\cdot,\cdot]_\g,l_3^\g)$ on a 2-term complex ${V}: \h_{{1}}\overset{\dM_\h}{\longrightarrow}\h_{0}$ is a Leibniz 2-algebra homomorphism $\rho^L=( l_0, l_1, l_2):\frkg\longrightarrow \End({V})$ and a map
$\rho^R=( r_0, r_1,  r_2, m_2):\frkg\longrightarrow \End({V})$
satisfying the following axioms:
\begin{eqnarray}
  \label{LLM1}&& l_0([x,y]_{\g})-[ l_0(x), l_0(y)]=\delta l_2(x,y),\\
  \label{LML2}&&  r_0([x,y]_{\g})-[ l_0(x), r_0(y)]=\delta m_2(x,y),\\
  \label{MLL3}&&  r_0([x,y]_{\g})- l_0(x) r_0(y) - r_0(y) r_0(x)=\delta r_2(x,y),\\[1em]
 &&   l_1([x,a]_{\g})-[ l_0(x), l_1(a)]= l_2(x,\dM_{\g}a),\\
 &&   r_1([x,a]_{\g})-[ l_1(x), r_1(a)]= m_2(x,\dM_{\g}a),\\
  &&  r_1([x,a]_{\g})- l_1(x) r_1(a) - r_1(a) r_0(x)= r_2(x,\dM_{\g}a),\\[1em]
  &&  l_0([a,x]_{\g})-[ l_0(a), l_0(x)]=  l_2(\dM_{\g}a,x)),\\
  &&  r_0([a,x]_{\g})-[ l_0(a), r_0(x)]=  m_2(\dM_{\g}a,x),\\
 &&   r_1([a,x]_{\g})- l_1(a) r_0(x) - r_0(x) r_1(a)= r_2(\dM_{\g}a,x),\\[1em]
\nonumber&& [ l_0(x),  l_2(y,z)] - [ l_2(x,y),  l_0(z)] - [ l_0(y),  l_2(x,z)]\\
&& + l_2(x, [y,z])-  l_2([x,y],z) -  l_2(y,[x,z])= l_1(l_3^\g(x,y,z)),\\[1em]
\nonumber && m_2(x, [y,z])- m_2([x,y],z) -  m_2(y,[x,z])\\
&&+[ l_0(x),  m_2(y,z)]-[ l_2(x,y), r_0(z)]-[ l_0(y), m_2(x,z)]= r_1(l_3^\g(x,y,z)).
\end{eqnarray}
\end{defi}


For any Leibniz 2-algebra $(\frkg;\dM_\g,[\cdot,\cdot]_\g,l_3^\g)$, there is a natural {\bf adjoint representation} on itself. The corresponding representation $\ad=(\ad^L_0,\ad^R_0,\ad^L_1,\ad^R_1,\ad^L_2,\ad^M_2,\ad^R_2)$ is given by
\begin{eqnarray*}
&&\ad^L_0(x)(y+b)=[x,y+b]_\g,\quad \ad^R_0(x)(y+b)=[y+b,x]_\g,\\
&&\ad^L_1(a)x=[a,x]_\g,\quad \ad^R_1(a)x=[x,a]_\g,\\
&&\ad^L_2(x,y)z=-l_3^\g(x,y,z),\quad \ad^M_2(x,y)z=-l_3^\g(x,z,y),\\
&&\ad^R_2(x,y)z=-l_3^\g(z,x,y).
\end{eqnarray*}

Next we develop a cohomology theory for a Leibniz 2-algebra $\fg$ with coefficients in a representation $({V},\rho)$.
The $k$-cochian  $C^{k}(\fg,{V})$ is defined to be the space:
\begin{eqnarray}
\Hom(\otimes^{k} \fg_0,V_0)\oplus\Hom(\mathcal{L}^{k-1},V_1)\oplus \Hom(\mathcal{L}^{k-2},V_0)\oplus \Hom(\fg_1^{k-1},V_1)\oplus \cdots
\end{eqnarray}
where
\begin{eqnarray*}
 &&\Hom(\mathcal{L}^{k-1},V_1):=\bigoplus_{i=1}^k  \Hom(\underbrace{{\fg_0}\otimes\cdots\otimes {\fg_0}}_{i-1}\otimes \fg_1\otimes \underbrace{{\fg_0}\otimes\cdots\otimes {\fg_0}}_{k-i},V_1),
 \end{eqnarray*}
is the direct sum of all possible tensor powers of ${\fg_0}$ and $\fg_1$ in which ${\fg_0}$ appears $k-1$ times but $\fg_1$ appears only once  and similarly for $\Hom(\mathcal{L}^{k-2},V_0).$

The cochain complex is given by
\begin{equation} \label{eq:complex}
\begin{split}
 &  V_0\stackrel{D_0}{\longrightarrow} \Hom(\fg_0,V_0)\oplus\Hom(\fg_1,V_1)\oplus\Hom(\otimes^2\frkg_{0},V_{{1}})\stackrel{D_1}{\longrightarrow}\\
 & \Hom(\fg_1,V_0)\oplus \Hom({\fg_0}\otimes {\fg_0}, V_0)\oplus \Hom(\mathcal{L}^{1}V_1)\oplus \Hom(\otimes^3\frkg_{0},V_{{1}})\stackrel{D_2}{\longrightarrow}\\
  &\Hom(\otimes^2\fg_1,V_0)\oplus \Hom(\mathcal{L}^{1},V_0)\oplus
   \Hom(\otimes^3 \fg_0, V_0)\oplus \Hom(\mathcal{L}^{2}, V_1)\oplus
      \Hom(\otimes^4 \fg_0, V_1)\\
  & \stackrel{D_3}{\longrightarrow}\cdots,
\end{split}
\end{equation}

Thus we get a cochain complex $C^k(\fg,{V})$ whose cohomology group
$$H^k(\fg,{V})=Z^k(\fg,{V})/B^k(\fg,{V})$$
is defined as the cohomology group of $\fg$ with coefficients in ${V}$.
We only investigate in detail the first and the second cohomology groups which is needed in the following of this paper as follows.


We give the precise formulas for the coboundary operator in low dimension as follows.


For any 1-cochain $\Phi=(\phi,\varphi,\chi)$, where $\phi\in\Hom(\frkg_{0},V_{0})$, $\varphi\in \Hom(\frkg_{{1}},V_{{1}})$,  $\chi\in \Hom(\otimes^2\frkg_{0},V_{{1}})$,
it is called a $1$-cocycle if the following equations hold:
\begin{eqnarray}
\label{eq:1-cocycle1}D_1\Phi(a)&=&\dM_\frkh \varphi(a)-\phi \dM_\g (a)=0,\\
\label{eq:1-cocycle2}D_1\Phi(x,y)&=& l_0(x)\phi(y)- r_0(y)\phi(x)-\phi[x,y]_\g+\dM_\frkh\circ \chi(x,y)=0,\\
\label{eq:1-cocycle3}D_1\Phi(x,a)&=& l_0(x)\varphi(a)- r_1(a)\phi(x)-\varphi[x,a]_\g+\chi(x,\dM_\g (a))=0,\\
\label{eq:1-cocycle32}D_1\Phi(a,x)&=& l_1(a)\varphi(x)- r_0(x)\phi(a)-\varphi[a,x]_\g+\chi(\dM_\g (a),x)=0,\\
\nonumber D_1\Phi(x,y,z)&=& l_0(x)\chi(y,z)- r_0(z)\chi(x,y)- l_0(y)\chi(x,z)\\
\nonumber&&-\varphi(l_3^\g(x,y,z)+ l_2(x,y)(\phi(z))+ m_2(x,z)(\phi(y))\\
&&+ r_2(y,z)(\phi(x))+\chi(x, [y,z])- \chi([x,y],z) - \chi(y,[x,z])=0.
\end{eqnarray}


For any 2-cochain $(\psi,\omega,\mu,\nu,\theta)$, where $\psi\in \Hom(\frkg_{{1}},V_{0})$, $\omega\in\Hom(\otimes^2\frkg_{0},V_{0})$,
$\mu\in\Hom(\frkg_{{1}}\otimes\frkg_{0},V_{{1}})$, $\nu\in\Hom(\frkg_{0}\otimes\frkg_{{1}},V_{{1}})$, $\theta\in \Hom(\otimes^3\frkg_{0},V_{{1}})$, it is called a 2-coboundary if $(\psi,\omega,\mu,\nu,\theta)=D_1(\phi,\varphi,\chi)$.
It is called a 2-cocycle if the following equalities hold:
\begin{eqnarray}
\label{eq:coc01}&& l_0(x)\psi(a)-\psi([x,a]_\g)+\omega(x,\dM_\g (a))-\dM_V\mu(x,a)=0,\\
\label{eq:coc02}&& r_0(x)\psi(a)-\psi([a,x]_\g)+\omega(\dM_\g (a),x)-\dM_V\nu(a,x)=0,\\
\label{eq:coc03}&& l_1(a)\psi(b)+\nu(a,\dM_\g (b))- r_1(b)\psi(a)-\mu(\dM_\g (a),b)=0,\\[1em]
\nonumber&& l_0(x)\omega(y,z)- r_0(z)\omega(x,y)- l_0(y)\omega(x,z)\\
\label{eq:coc04}&&+\omega(x,[y,z]_\g)-\omega([x,y]_\g,z)-\omega(y,[x,z]_\g)-\partial\circ\theta(x,y,z)-\psi(l_3^\g(x,y,z))=0,\\[1em]
\nonumber&& l_0(x)\mu(y,a)- r_0(a)\omega(x,y)- l_0(y)\mu(x,a)\\
\label{eq:coc05}&&+\mu(x,[y,a]_\g)-\mu([x,y]_\g,a)-\mu(y,[x,a]_\g)-\theta(x,y,\dM_\g (a))+ l_2(x,y)\psi(a)=0,\\[1em]
\nonumber&& l_0(x)\nu(a,y)- r_0(y)\omega(x,a)- l_1(a)\omega(x,y)\\
\label{eq:coc06}&&+\mu(x,[a,y]_\g)-\nu([x,a]_\g,y)-\nu(a,[x,y]_\g)-\theta(x,\dM_\g (a),y)+ m_2(x,y)\psi(a)=0,\\[1em]
\nonumber&& l_1(a)\omega(x,y)- r_0(y)\omega(a,x)- l_0(x)\nu(a,y)\\
\label{eq:coc07}&&+\nu(a,[x,y]_\g)-\nu([a,x]_\g,y)-\mu(x,[a,y]_\g)-\theta(\dM_\g (a),x,y)+ r_2(x,y)\psi(a)=0,\\[1em]
\nonumber&& l_0(x)\theta(y,z,t)- l_0(y)\theta(x,z,t)+ l_0(z)\theta(x,y,t)+ r_0(t)\theta(x,y,z)\\
\nonumber&&-\theta([x,y]_\g,z,t)-\theta(y,[x,z]_\g,t)-\theta(y,z,[x,t]_\g)\\
\nonumber&&+\theta(x,[y,z]_\g,t)+\theta(x,z,[y,t]_\g)-\theta(x,y,[z,t]_\g)\\
\nonumber&& -\{\mu(x,l_3^\g(y,z,t))-\mu(y,l_3^\g(x,z,t))+\mu(z,l_3^\g(x,y,t))+\nu(l_3^\g(x,y,z),t)\\
\nonumber&&- r_2(z,t)\omega(x,y)- m_2(y,t)\omega(x,z)- l_2(y,z)\omega(x,t)\\
\label{eq:coc08}&&+ m_2(x,t)\omega(y,z)+ l_2(x,z)\omega(y,t)- l_2(x,y)\omega(z,t)\}=0.
\end{eqnarray}
The spaces of 2-coboundaries and 2-cocycles are denoted by $B^2(\fg, {V})$ and $Z^2(\fg, {V})$ respectively.
By direct computations, it is easy to see that the space of 2-coboundaries is contained in space of 2-cocycles, thus we obtain a second cohomology group  of a Leibniz 2-algebra $\fg$ with coefficients in ${V}$  defined to be the quotient space
$$H^2(\fg,  {V}))\triangleq Z^2(\fg, {V})/B^2(\fg, {V}).$$

\begin{defi}
Let $(\fg;\dM_\g,[\cdot,\cdot]_\g,l_3^\g)$,  $(\widehat{\mathfrak{g}};\widehat{\dM},[\cdot,\cdot]_{\widehat{\g}},\widehat{l_3})$ be Leibniz 2-algebras and
$i=(i_{0},i_{1}):\h\longrightarrow\widehat{\mathfrak{g}},~~p=(p_{0},p_{1}):\widehat{\mathfrak{g}}\longrightarrow\mathfrak{g}$
be strict homomorphisms. The following sequence of Leibniz 2-algebras is a
short exact sequence if $\mathrm{Im}(i)=\mathrm{Ker}(p)$,
$\mathrm{Ker}(i)=0$ and $\mathrm{Im}(p)=\g$.

\begin{equation}\label{eq:ext2}
\CD
  0 @>0>>  \h_{{1}} @>i_1>> \widehat{\g}_{{1}} @>p_1>> \g_{{1}} @>0>> 0 \\
  @V 0 VV @V \dM_\h VV @V \widehat{\dM} VV @V\dM_\g VV @V0VV  \\
  0 @>0>> \h_{0} @>i_0>> \widehat{\g}_0 @>p_0>> \g_0@>0>>0
\endCD
\end{equation}
We call $\widehat{\mathfrak{g}}$  an extension of $\mathfrak{g}$ by
$\h$, and denote it by $\E_{\widehat{\g}}.$
It is called an abelian extension if $\h$ is abelian, i.e. if $[\cdot,\cdot]_{\h}=0$ and $l^{\h}_3(\cdot,\cdot,\cdot)=0$. We will view $\h$ as subcomplex of $\g$ directly, and omit the map $i$.
\end{defi}

 A splitting $\sigma:\mathfrak{g}\longrightarrow\widehat{\mathfrak{g}}$ of $p:\widehat{\mathfrak{g}}\longrightarrow\mathfrak{g}$
consists of linear maps
$\sigma_0:\mathfrak{g}_0\longrightarrow\widehat{\g}_0$ and
$\sigma_1:\mathfrak{g}_{{1}}\longrightarrow\widehat{\g}_{{1}}$
 such that  $p_0\circ\sigma_0=id_{\mathfrak{g}_0}$ and  $p_1\circ\sigma_1=id_{\mathfrak{g}_{{1}}}$.

 Two extensions of Leibniz 2-algebras
 $\E_{\widehat{\g}}:0\longrightarrow\frkh\stackrel{i}{\longrightarrow}\widehat{\g}\stackrel{p}{\longrightarrow}\g\longrightarrow0$
 and $\E_{\tilde{\g}}:0\longrightarrow\frkh\stackrel{j}{\longrightarrow}\tilde{\g}\stackrel{q}{\longrightarrow}\g\longrightarrow0$ are equivalent,
 if there exists a Leibniz $2$-algebra homomorphism $F:\widehat{\g}\longrightarrow\tilde{\g}$  such that $F\circ i=j$, $q\circ
 F=p$ and $F_2(i(u),\alpha)=0$, for any
 $u\in\frkh_0,~\alpha\in\widehat{\g}_0$.

Let $\widehat{\mathfrak{g}}$ be an abelian extension of $\mathfrak{g}$ by
$\frkh$, and $\sigma:\g\longrightarrow\widehat{\g}$ be a splitting.
Define $\mu=(\rho_0,\rho_1,\rho_2)$ by
\begin{equation}\label{eq:morphism1}
\left\{\begin{array}{rlclcrcl}
l_0,r_0:&\mathfrak{g}_{0}&\longrightarrow& \End^{0}_{\dM_\frkh}(\frkh),&& l_0(x)(u+m)&\triangleq&[\sigma_0(x),u+m]_{\hg},\\
                                                  && & && r_0(x)(u+m)&\triangleq&[u+m, \sigma_0(x)]_{\hg},\\
l_1,r_1:&\mathfrak{g}_{{1}}&\longrightarrow& \End^{1}(\frkh),&&l_1(a)(u)&\triangleq&[\sigma_1(a),u]_{\hg},\\
                                                  &&& &&r_1(a)(u)&\triangleq&[u,\sigma_1(a)]_{\hg},\\
l_2,m_2,r_2:&\mathfrak{g}_{0}\times\mathfrak{g}_{0}&\longrightarrow&\End^{1}(\frkh),&& l_2(x,y)(u)&\triangleq&-\widehat{l_{3}}(\sigma_0(x),\sigma_0(y),u),\\
&&& && m_2(x,y)(u)&\triangleq&-\widehat{l_{3}}(\sigma_0(x),u,\sigma_0(y)),\\
&&& && r_2(x,y)(u)&\triangleq&-\widehat{l_{3}}(u,\sigma_0(x),\sigma_0(y)),
\end{array}\right.
\end{equation}
for any $x,y\in\mathfrak{g}_{0}$, $a\in\mathfrak{g}_{{1}}$,
$u\in\frkh_{0}$ and $m\in\frkh_{{1}}$.

\begin{prop}[\cite{Zhang}]\label{pro:2-modules2}
With the above notations, $\rho=(l_0,r_0,l_1,r_1,l_2,m_2,r_2)$ is a representation of Leibniz $2$-algebra $\g$ on $\h$. Furthermore,  $\rho$ does not depend on the choice of the splitting $\sigma$.
Moreover,  equivalent abelian extensions give the same representation of $\g$ on $\h$.
\end{prop}

Let $\sigma:\frkg\longrightarrow\widehat{\frkg}$  be a
splitting of the abelian extension \eqref{eq:ext2}. Define the following linear maps:
$$
\begin{array}{rlclcrcl}
\psi:&\frkg_{{1}}&\longrightarrow&\h_{0},&& \psi(a)&\triangleq&\widehat{\dM}\si(a)-\si(\dM_\frkg (a)),\\
\omega:&\otimes^2\frkg_{0}&\longrightarrow&\h_{0},&& \omega(x,y)&\triangleq&[\si(x),\si(y)]_{\widehat{\frkg}}-\si[x,y]_\frkg,\\
\mu:&\frkg_{0}\otimes\frkg_{{1}}&\longrightarrow&\h_{{1}},&& \mu(x,a)&\triangleq&[\si(x),\si(a)]_{\widehat{\frkg}}-\si[x,a]_\frkg,\\
\nu:&\frkg_{{1}}\otimes\frkg_{0}&\longrightarrow&\h_{{1}},&& \nu(a,x)&\triangleq&[\si(a),\si(x)]_{\widehat{\frkg}}-\si[a,x]_\frkg,\\
\theta:&\otimes^3\frkg_{0}&\longrightarrow&\h_{{1}},&&
\theta(x,y,z)&\triangleq&\widehat{l}_{3}(\si(x),\si(y),\si(z))-\si(l_{3}^\frkg(x,y,z)),
\end{array}
$$
for any $x,y,z\in\frkg_{0}$, $a\in\frkg_{{1}}$.

\begin{thm}[\cite{Zhang}]\label{thm:2-cocylce}
Let $\E_{\hg}$ is an abelian extension of $\g$ by $\h$ given by \eqref{eq:ext2}, then $(\psi,\omega,\mu,\nu,\theta)$ is a $2$-cocycle of $\g$ with coefficients in $\frkh$, where the representation is given by $\rho$ above defined.
\end{thm}

Now we can construct a Leibniz 2-algebra structure  on $\g\oplus\frkh$ using the 2-cocycle given above. More precisely, we have
\begin{equation}\label{eq:bracket}
\left\{\begin{array}{rcl}
\dM_{\gh}(a+m)&\triangleq&\dM_\g(a)+\psi(a)+\dM_\frkh(a),\\
\brh{x+u,y+v}&\triangleq&[x,y]_\g+\omega(x,y)+l_0(x)v+r_0(y)u,\\
\brh{x+u,a+m}&\triangleq&[x,a]_\g+\mu(x,a)+l_0(x)m+r_1(a)u,\\
\brh{a+m,x+u}&\triangleq&[a,x]_\g+\nu(a,x)+l_1(a)u+r_0(x)m,\\
l_3^{\gh}(x+u,y+v,z+w)&\triangleq& l_3^\g(x,y,z)+\theta(x,y,z)\\
&&-l_2(x,y)(w)-m_2(x,z)(v)-r_2(y,z)(u),
\end{array}\right.
\end{equation}
for any $x,y,z\in\mathfrak{g}_{0}$, $a\in\mathfrak{g}_{{1}}$,
$u\in\frkh_{0}$ and $m\in\frkh_{{1}}$. Thus any extension $E_{\widehat{\g}}$ given by \eqref{eq:ext2} is
isomorphic to
\begin{equation}\label{ext2}
\CD
  0 @>0>>  \frkh_{{1}} @>i_1>> \g_{{1}}\oplus \frkh_{{1}} @>p_1>> \g_{{1}} @>0>> 0 \\
  @V 0 VV @V \dM_\frkh VV @V \widehat{\dM} VV @V\dM_\g VV @V0VV  \\
  0 @>0>> \frkh_{0} @>i_0>> \g_0\oplus \frkh_0 @>p_0>> \g_0@>0>>0,
\endCD
\end{equation}
where the Leibniz 2-algebra structure on $\frkg\oplus\frkh$ is given by
\eqref{eq:bracket},
$(i_0,i_1)$ is the inclusion and $(p_0,p_1)$ is the projection.

\begin{thm}[\cite{Zhang}]\label{mainthm44}
There is a one-to-one correspondence between equivalence classes of abelian extensions of the Leibniz 2-algebras $\frkg$ by $\h$ and the second cohomology group $H^2(\frkg;{V})$.
\end{thm}

\section{Wells exact sequence for  automorphisms}
In this section, we consider the inducibility of a pair of automorphisms in abelian extensions of Leibniz 2-algebras. Theorem \ref{main-thm2} provides
necessary and sufficient conditions for a pair of automorphisms to be inducible.

Let $V$ and $\fg$ be Leibniz 2-algebras. Then an extension of $\fg$ by $V$ is a short exact sequence of Leibniz 2-algebras $$0 \to V \stackrel{i}{\to} \widehat{\g} \stackrel{p}{\to} \g \to 0,$$ where $\widehat{\g}$ is a Leibniz 2-algebra. Without loss of generality, we may assume that $i$ is the inclusion map and we omit it from the notation.

Let $\Aut(V)$, $\Aut(\widehat{\g})$ and $\Aut(\fg)$ denote the groups of all Leibniz 2-algebra automorphisms of $V$, $\widehat{\g}$ and $\fg$, respectively.  Let $\Aut_V(\widehat{\g})$ denote the group of all Leibniz 2-algebra automorphisms of $\widehat{\g}$ which keep $V$ invariant as a set. Note that an automorphism ${F}=({F}_0, {F}_1, {F}_2) \in \Aut_V(\widehat{\g})$ induces automorphisms ${F}|_V=(\beta_0, \beta_1) \in \Aut(V)$ given by
 $$\beta_0(v)={F}_0|_V(v),\quad \beta_1(m)={F}_1|_V(m),$$
 for all $u, v \in V_0, m \in V_1$.
Next we  define a map $\overline{{F}}=(\alpha_0, \alpha_1, \alpha_2) : \mathfrak{g} \rightarrow \mathfrak{g}$ by
\begin{align*}
\alpha_0(x)=p_0{F}_0\sigma_0(x),\quad \alpha_1(a)=p_1{F}_1\sigma_1(a), \quad \alpha_2(x, y)=p_1{F}_2(\sigma_0(x),\sigma_0(y)),
\end{align*}
for all $x, y \in \g_0, a \in \g_1$.

\begin{prop}
Let  ${F}=({F}_0, {F}_1, {F}_2)$  be a Leibniz 2-algebra  automorphism of $\hg$. Then   $\overline{{F}}=(\alpha_0, \alpha_1, \alpha_2)$ defined above is  a Leibniz 2-algebra  automorphism of $\g$.
\end{prop}
\begin{proof}
 Since ${F}$ is an automorphism on $\widehat{\g}$ preserving the space $V$, it also preserves the subspace $\mathfrak{g}$, it is easy to see that $\al_0$ and $\al_1$ are bijective maps on $\fg_0$ and $\fg_1$.

For any $x,y \in \g_0, u, v \in \g_1$, we have
\begin{align*}
(\alpha_0\dM_{\g}-\dM_{\g}\alpha_1)(a)=&(p_0{F}_0\sigma_0\dM_{\g}-\dM_{\g}p_1{F}_1\sigma_1)(a)\\
=&p_0{F}_0\sigma_0\dM_{\g}-p_1{F}_1\widehat{\dM}\sigma_1)(a)\\
=&p_0{F}(\sigma_0\dM_{\g}-\widehat{\dM}\sigma_1)(a)\\
=&0,\\
\dM_{\g}\alpha_2(x,y)=&\dM_{\g}p_1{F}_2(\sigma_0(x), \sigma_0(y))\\
=&p_0\widehat{\dM}{F}_2(\sigma_0(x), \sigma_0(y))\\
=&p_0(\alpha_0([\sigma_0(x), \sigma_0(y)])-[{F}0(\sigma(x)), {F}0(\sigma(y))])\\
=&p_0\alpha_0(\omega(x, y)+ \sigma_0([x,y]))-[\alpha_0(x), \alpha_0(y)]\\
=&\alpha_0([x,y])-[\alpha_0(x), \alpha_0(y)].
\end{align*}

Similarly, we have
\begin{align*}
\alpha_2(x, \dM_{\g}(a))=&\alpha_1[x,a]-[\alpha_0(x),\alpha_1(a)],\\
\alpha_2(\dM_{\g}(a),x)=&\alpha_1[a, x]-[\alpha_1(a), \alpha_0(x)].
\end{align*}
Next, we have
\begin{align*}
&\alpha_1(l_3^\g(x,y,z))-l_3^\g(\alpha_0(x),\alpha_0(y),\alpha_0(z))\\
=&\alpha_1(\widehat{l}_3(\sigma_0(x),\sigma_0(y),\sigma_0(z))-\theta(x,y,z))\\
&~-l_3^\g(p_0{F}_0\sigma_0(x),p_0{F}_0\sigma_0(y),p_0{F}_0\sigma_0(z))\\
=&p_1\big({F}_1\widehat{l}_3(\sigma_0(x),\sigma_0(y),\sigma_0(z))\\
&~-\widehat{l}_3({F}_0\sigma_0(x),{F}_0\sigma_0(y),{F}_0\sigma_0(z))\big)\\
=&p_1{F}_2(\sigma_0(x), [\sigma_0(y),\sigma_0(z)])- p_1{F}_2([\sigma_0(x),\sigma_0(y)],\sigma_0(z))\\
 &~- p_1{F}_2(\sigma_0(y),[\sigma_0(x),\sigma_0(z)])+ p_1[{F}_0(\sigma_0(x)), {F}_2(\sigma_0(y),\sigma_0(z))]\\
  &~- p_1[{F}_2(\sigma_0(x),\sigma_0(y)), {F}_0(\sigma_0(z))] - p_1[{F}_0(\sigma_0(y)), {F}_2(\sigma_0(x),\sigma_0(z))]\\
 =&\alpha_2(x, [y,z])- \alpha_2([x,y],z) - \alpha_2(y,[x,z])\\
 &~+[\alpha_0(x), \alpha_2(y,z)] - [\alpha_2(x,y), \alpha_0(z)] - [\alpha_0(y), \alpha_2(x,z)].
\end{align*}
This shows that $\overline{{F}} : \mathfrak{g} \rightarrow \mathfrak{g}$ is a Leibniz 2-algebra automorphism.
\end{proof}

 Now we obtain a group homomorphism
$$\Phi: \Aut_V(\widehat{\g}) \to \Aut(V) \times \Aut(\fg)$$
given by
$$\Phi({F})= ({F}|_V, \overline{{F}}).$$

A pair of automorphisms $(\beta, \alpha) \in \Aut(V)\times \Aut(\fg)$ is called  \textbf{inducible} if there exists a ${F} \in \Aut_V(\widehat{\g})$ such that
$$\Phi({F}_0)= (\beta_0, \alpha_0),\quad \Phi({F}_1)= (\beta_1, \alpha_1),\quad \Phi({F}_2)= (0, \alpha_2).$$

\begin{thm}\label{main-thm2}
Let $0 \to V \to \hg \to \fg \to 0$ be an abelian extension of Leibniz 2-algebras   and $(\beta, \alpha) \in \Aut(V)\times \Aut(\fg)$. Then the pair $(\beta, \alpha)$ is inducible if and only if exists a linear map $\lambda=(\lambda_0,\lambda_1,\lambda_2) \in \mathrm{Hom}(\mathfrak{g}, V)$ satisfying the followings
{\footnotesize
\begin{align}
\label{COC1}&\beta_0(\psi (a)) - \psi (\alpha_1(a)) = \dM_{V}(\lambda_1(a))-\lambda_0 (\dM_{\g}(a)),\\
&\beta_0 (\omega (x, y)) - \omega (\alpha_0(x), \alpha_0(y)) =  l_0({\alpha_0(x)}) (\lambda_0(y)) + \label{COC2}r_0({\alpha_0 (y)}) (\lambda_0(x))  - \lambda_0 ([x, y])+\dM_{V}\lambda_2(x,y)  \nonumber\\
&+\psi(\alpha_2(x,y)),\\
\label{COC3}&\beta_1 (\mu (x, a)) - \mu (\alpha_0(x), \alpha_1(a)) =  l_0({\alpha_ 0(x)})( \lambda_1(a)) + r_1({\alpha_1 (a)})( \lambda_0(x) ) - \lambda_1 ([x, a])+\lambda_2(x,\dM_{\g}(a)),\\
\label{COC4}&\beta_1 (\nu (a, x)) - \nu (\alpha_1(a), \alpha_0(x)) =  r_0({\alpha_0 (x)})( \lambda_1(a)) + l_1({\alpha_1 (a)})( \lambda_0(x) ) - \lambda_1 ([a, x])+\lambda_2(\dM_{\g}(a),x),\\
\label{COC5}&\beta_1\theta(x,y,z)-\theta(\alpha_0(x),\alpha_0(y),\alpha_0(z))-\mu(\alpha_0(x),\alpha_2(y,z))+\nu(\alpha_2(x,y),\alpha_0(z))+r_1(\alpha_2(x,z))(\lambda_0(y))\nonumber\\
&+\mu(\alpha_0(y),\alpha_2(x,z))= -l_2(\alpha_0(x),\alpha_0(y))(\lambda_0(z))- m_2(\alpha_0(x),\alpha_0(z))(\lambda_0(y))-\lambda_1(l^{\g}_3(x,y,z)) \nonumber\\
&- r_2(\alpha_0(y),\alpha_0(z))(\lambda_0(x))+\lambda_2(x,[y,z])-\lambda_2([x,y],z)-\lambda_2(y,[x,z])+ l_0(\alpha_0(x))(\lambda_2(y,z))\nonumber\\
&+ r_1(\alpha_2(y,z))(\lambda_0(x))- r_0(\alpha_0(z))(\lambda_2(x,y))
- l_1(\alpha_2(x,y))(\lambda_0(z))- l_0(\alpha_0(y))(\lambda_2(x,z)),\\
\label{COC6}&\beta_0 { l_0}(x){\beta_0}^{-1} (v) ={ l_0}({\alpha_0 (x)})(v),\\
\label{COC7}&\beta_0 { r_0}(x){\beta_0}^{-1} (v) = { r_0}({\alpha_0 (x)})  (v),\\
\label{COC8}&\beta_1 l_0(x){\beta_1}^{-1}(m)= l_0(\alpha_0(x))(m),\\
\label{COC9}&\beta_1 r_0(x){\beta_1}^{-1}(m)= r_0(\alpha_0(x))(m),\\
\label{COC10}&\beta_1 r_1(a){\beta_0}^{-1}(v)= r_1(\alpha_1(a))(v),\\
\label{COC11}&\beta_1 l_1(a){\beta_0}^{-1}(v)= l_1(\alpha_1(a))(v),\\
\label{COC12}&l_2(\alpha_0(x),\alpha_0(y))(w)+ l_1\alpha_2(x,y)(w)=\beta_1 l_2(x,y){\beta_0}^{-1}(w),\\
\label{COC13}&r_2(\alpha_0(x),\alpha_0(y))(w)- r_1\alpha_2(x,y)(w)=\beta_1 r_2(x,y){\beta_0}^{-1}(w),\\
\label{COC14}&m_2(\alpha_0(x),\alpha_0(y))(w)+ r_1\alpha_2(x,y)(w)=\beta_1 m_2(x,y){\beta_0}^{-1}(w).
\end{align}
}
\end{thm}

\begin{proof}
Suppose that there exists a ${F} \in \Aut_V(\hg)$ such that $\Phi({F})=(\beta, \alpha)$. We observe that
\begin{align*}
p_0 \big(  ({F}_0 \si_0 - \si_0 \alpha_0)(x)  \big) =&~ \alpha_0 (x) - \alpha_0 (x) = 0 ~~~~ (\text{as } p_0\si_0 = \mathrm{id}_{\mathfrak{g}_0}),\\
p_1 \big(  ({F}_1 \si_1 - \si_1 \alpha_1)(x)  \big) =&~ \alpha_1 (x) - \alpha_1 (x) = 0 ~~~~ (\text{as } p_1\si_1 = \mathrm{id}_{\mathfrak{g}_1}),\\
p_1 \big(  ({F}_2(\si_0,\si_0) - \si_1 \alpha_2)(x,y)  \big) =&~ \alpha_2 (x) - \alpha_2(x) = 0 ~~~~(\text{as } p_1\si_1 = \mathrm{id}_{\mathfrak{g}_1}).
\end{align*}
We define a map $\lambda=(\lambda_0,\lambda_1,\lambda_2): \mathfrak{g} \rightarrow V$ by
\begin{align*}
\lambda_0(x)=&~{F}_0 \sigma_0(x)-\sigma_0\alpha_0(x),\\
\lambda_1(a)=&~{F}_1 \sigma_1(a)-\sigma_1\alpha_1(a),\\
\lambda_2(x,y)=&~{F}_2(\sigma_0(x),\sigma_0(y))-\sigma_1\alpha_2(x,y).
\end{align*}
We utilize the property that ${F}$ is a homomorphism of Leibniz 2-algebras. Let $e_1=\si_0(x)+u$, $e_2=\si_0(y)+v$, $e_3=\si_0(z)+w$ and $h_1=\si_1(a)+m$, for all $x, y, z \in \g_0$, $u,v \in V_0$, $a \in \g_1$ and $ m\in V_1$. Then
\begin{align}
&\label{homo1}{F}_0\widehat{\dM}(h_1)=\widehat{\dM}{F}_1(h_1),\\
&\label{homo2}{F}_{0}[e_1,e_2]_\g-[{F}_{0}(e_1),{F}_{0}(e_2)]=\widehat{\dM}{F}_{2}(e_1,e_2),\\
&\label{homo3}{F}_{1}[e_1,h_1]_\g-[{F}_{0}(e_1),{F}_{1}(h_1)]={F}_{2}(e_1,\widehat{\dM}_\g (h_1)),\\
&\label{homo4}{F}_{1}[h_1,e_1]_\g-[{F}_{1}(h_1),{F}_{0}(e_1)]={F}_{2}(\widehat{\dM} (h_1),e_1),\\
&\label{homo5}{F}_1(\widehat{l_3}(e_1,e_2,e_3))-\widehat{l_3}({F}_0(e_1),{F}_0(e_2),{F}_0(e_3))={F}_2(e_1, [e_2,e_3])- {F}_2([e_1,e_2],e_3)\nonumber\\
  & - {F}_2(e_2,[e_1,e_3])+ [{F}_0(e_1), {F}_2(e_2,e_3)]' - [{F}_2(e_1,e_2), {F}_0(e_3)] - [{F}_0(e_2),{F}_2(e_1,e_3)].
\end{align}

First, we compute
\begin{align*}
{F}_0\big( l_0(x)(v)\big)={F}_0\big([\si_0(x),v]\big)=&[{F}_0\si_0(x),{F}_0(v)]+\widehat{\dM}{F}_2(\si_0(x),v)\\
=&[\lambda_0(x)+\si_0\alpha_0(x),\beta_0(v)]= l_0(\alpha_0(x))(\beta_0(v)),
\end{align*}
\begin{align*}
{F}_0\big( r_0(x)(v)\big)={F}_0\big([v,\si_0(x)]\big)=&[{F}_0(v),{F}_0\si_0(x)]+\widehat{\dM}{F}_2(v,\si_0(x))\\
=&[\beta_0(v),\lambda_0(x)+\si_0\alpha_0(x)]= r_0(\alpha_0(x))(\beta_0(v)).
\end{align*}
\begin{align*}
&{F}_1 l_2(x,y)(w)\\
=&-{F}_1\widehat{l_3}(\si_0(x),\si_0(y),w)\\
=&\widehat{l_3}({F}_0\si_0(x),{F}_0\si_0(y),{F}_0(w))-{F}_2(\si_0(x), [\si_0(y),w])+ {F}_2([\si_0(x),\si_0(y)],w)\\
&~+ {F}_2(\si_0(y),[\si_0(x),w])- [{F}_0\si_0(x), {F}_2(\si_0(y),w)] + [{F}_2(\si_0(x),\si_0(y)), {F}_0(w)] \\
&~+ [{F}_0\si_0(y), {F}_2(\si_0(x),w)]\\
=&\widehat{l_3}({F}_0\si_0(x),{F}_0\si_0(y),{F}_0(w))+ [{F}_2(\si_0(x),\si_0(y)), {F}_0(w)]\\
=&\widehat{l_3}(\lambda_0(x)+\si_0\alpha_0(x),\lambda_0(y)+\si_0\alpha_0(y),\beta_0(w))+[\lambda_2(x,y)+\si_1\alpha_2(x,y),\beta_0(w)]\\
=&-\widehat{l_3}(\si_0\alpha_0(x),\si_0\alpha_0(y),\beta_0(w))+[\lambda_2(x,y)+\si_1\alpha_2(x,y),\beta_0(w)]\\
=& l_2(\alpha_0(x),\alpha_0(y))(\beta_0(w))+ l_1\alpha_2(x,y)(\beta_0(w)).
\end{align*}
Similarly, we have
\begin{align*}
{F}_1 r_2(x,y)(w)=r_2(\alpha_0(x),\alpha_0(y))(\beta_0(w))- r_1\alpha_2(x,y)(\beta_0(w)),\\
{F}_1 m_2(x,y)(w)= m_2(\alpha_0(x),\alpha_0(y))(\beta_0(w))+ r_1\alpha_2(x,y)(\beta_0(w)).
\end{align*}

For \eqref{homo1}, we have
\begin{align*}
&{F}_0\widehat{\dM}(\sigma_1(a)+m))\\
=&{F}_0(\psi(a)+\sigma_0\dM_\g(a)+\dM_V(m))\\
=&\beta_0\psi(a)+\lambda_0(\dM_\g(a))+\beta_0\dM_V(m)+\alpha_0\sigma_0\dM_\g(a),\\
\\
&\widehat{\dM}{F}_1(\sigma_1(a)+m)\\
=&\widehat{\dM}(\lambda_1(a)+\sigma_1(\alpha_1(a))+\beta_1(m))\\
=&\psi(\alpha_1(a))+\sigma_0\dM_\g\alpha_1(a)+\dM_V(\lambda_1(a))+\dM_V(\beta_1(m)).
\end{align*}
Thus, the two sides are equal if and only if condition (\ref{COC1}) holds.

For \eqref{homo2}, we get
\begin{align*}
&{F}_{0}[\sigma_0(x)+u,\sigma_0(y)+v]-[{F}_{0}(\sigma_0(x)+u),{F}_{0}(\sigma_0(y)+v)]\\
=&\beta_0\omega(x,y)+\sigma_0\alpha_0[x,y]+\lambda_0[x,y]+\beta_0 l_0(x)(v)+\beta_0 r_0(y)(u)\\
&~- l_0(\alpha_0(x))(\lambda_0(y))- l_0(\alpha_0(x))(\beta_0(v))- r_0(\alpha_0(y))(\lambda_0(x))\\
&~- r_0(\alpha_0(y))(\beta_0(u))-\omega(\alpha_0(x),\alpha_0(y))-\sigma_0[\alpha_0(x),\alpha_0(y)],\\
\\
&\widehat{\dM}{F}_{2}(\sigma_0(x)+u,\sigma_0(y)+v)\\
=&\psi(\alpha_2(x,y))+\sigma_0\dM_\g\alpha_2(x,y)+\dM_V\lambda_2(x,y)
\end{align*}
Thus, the two sides are equal if and only if condition (\ref{COC2}) holds.

Similarly, from \eqref{homo3} and \eqref{homo4}, we derive conditions \eqref{COC3} and \eqref{COC4}.

For \eqref{homo5}, the left hand side is equal to
\begin{align*}
&{F}_1(\widehat{l_3}(e_1,e_2,e_3))-\widehat{l_3}({F}_0(e_1),{F}_0(e_2),{F}_0(e_3))\\
=&{F}_1(\widehat{l_3}(\sigma_0(x)+u,\sigma_0(y)+v,\sigma_0(z)+w))-\widehat{l_3}({F}_0(\sigma_0(x)+u),{F}_0(\sigma_0(y)+v),{F}_0(\sigma_0(z)+w))\\
=&{F}_1\widehat{l_3}(\sigma_0(x),\sigma_0(y),\sigma_0(z))+\beta_1\widehat{l_3}(\sigma_0(x),\sigma_0(y),w)+\beta_1\widehat{l_3}(\sigma_0(x),v,\sigma_0(z))\\
&~+\beta_1\widehat{l_3}(u,\sigma_0(y),\sigma_0(z))-\widehat{l_3}(\sigma_0\alpha_0(x),\sigma_0\alpha_0(y),\sigma_0\alpha_0(z))-\widehat{l_3}(\sigma_0\alpha_0(x),\sigma_0\alpha_0(y),\lambda_0(z))\\
&~-\widehat{l_3}(\sigma_0\alpha_0(x),\sigma_0\alpha_0(y),\beta_0(w))-\widehat{l_3}(\sigma_0\alpha_0(x),\lambda_0(y),\sigma_0\alpha_0(z))-\widehat{l_3}(\sigma_0\alpha_0(x),\beta_0(v),\sigma_0\alpha_0(z))\\
&~-\widehat{l_3}(\lambda_0(x),\sigma_0\alpha_0(y),\sigma_0\alpha_0(z))-\widehat{l_3}(\beta_0(u),\sigma_0\alpha_0(y),\sigma_0\alpha_0(z))\\
=&\sigma_1\alpha_1l^{\g}_3(x,y,z)+\lambda_1(l^{\g}_3(x,y,z))-\sigma_1l^\g_3(\alpha_0(x),\alpha_0(y),\alpha_0(z))+\beta_1\theta(x,y,z)-\beta_1 r_2(y,z)(u)\\
&~+ l_2(\alpha_0(x),\alpha_0(y))(\lambda_0(z))+ m_2(\alpha_0(x),\alpha_0(z))(\lambda_0(y))+ r_2(\alpha_0(y),\alpha_0(z))(\lambda_0(x))\\
&~+ l_2(\alpha_0(x),\alpha_0(y))(\beta_0(w))-\beta_1 l_2(x,y)(w)+ m_2(\alpha_0(x),\alpha_0(z))(\beta_0(v))-\beta_1 m_2(x,z)(v)\\
&~+ r_2(\alpha_0(y),\alpha_0(z))(\beta_0(u))-\theta(\alpha_0(x),\alpha_0(y),\alpha_0(z)),
\end{align*}
and the right hand side is equal to
\begin{align*}
&{F}_2(e_1, [e_2,e_3])- {F}_2([e_1,e_2],e_3) - {F}_2(e_2,[e_1,e_3])+ [{F}_0(e_1), {F}_2(e_2,e_3)]\\
  & - [{F}_2(e_1,e_2), {F}_0(e_3)] - [{F}_0(e_2),{F}_2(e_1,e_3)]\\
=&{F}_2(\sigma_0(x)+u, [\sigma_0(y)+v,\sigma_0(z)+w])- {F}_2([\sigma_0(x)+u,\sigma_0(y)+v],\sigma_0(z)+w)\\
 &~- {F}_2(\sigma_0(y)+v,[\sigma_0(x)+u,\sigma_0(z)+w])
+ [{F}_0(\sigma_0(x)+u), {F}_2(\sigma_0(y)+v,\sigma_0(z)+w)] \\
&~- [{F}_2(\sigma_0(x)+u,\sigma_0(y)+v), {F}_0(\sigma_0(z)+w)] - [{F}_0(\sigma_0(y)+v), {F}_2(\sigma_0(x)+u,\sigma_0(z)+w)]\\
=&{F}_2(\sigma_0(x)+u, \omega(y,z)+\sigma_0[y,z]+ l_0(y)(w)+ r_0(z)(v))\\
&~-{F}_2(\omega(x,y)+\sigma_0[x,y]+ l_0(x)(v)+ r_0(y)(u), \sigma_0(z)+w)\\
&~-{F}_2(\sigma_0(y)+v, \omega(x,z)+\sigma_0[x,z]+ l_0(x)(w)+ r_0(z)(y))\\
&~+[\lambda_0(x)+\sigma_0\alpha_0(x)+\beta_0(u),\lambda_2(y,z)+\sigma_1\alpha_2(y,z)]\\
&~-[\lambda_2(x,y)+\sigma_1\alpha_2(x,y),\lambda_0(z)+\sigma_0\alpha_0(z)+\beta_0(w)]\\
&~-[\lambda_0(y)+\sigma_0\alpha_0(y)+\beta_0(v),\lambda_2(x,z)+\sigma_1\alpha_2(x,z)]\\
=&\lambda_2(x,[y,z])+\sigma_1\alpha_2([x,y],z)-\lambda_2([x,y],z)-\sigma_1\alpha_2(x,[y,z])-\lambda_2(y,[x,z])\\
&~-\sigma_1\alpha_2(y,[x,z])+ l_0(\alpha_0(x))(\lambda_2(y,z))+ r_1(\alpha_2(y,z))(\lambda_0(x))+ r_1(\alpha_2(y,z))(\beta_0(v))\\
&~+\mu(\alpha_0(x),\alpha_2(y,z))+\si_1[\alpha_0(x),\alpha_2(y,z)]- r_0(\alpha_0(z))(\lambda_2(x,y))- l_1(\alpha_2(x,y))(\lambda_0(z))\\
&~- l_1(\alpha_2(x,y))(\beta_0(w))-\nu(\alpha_2(x,y),\alpha_0(z))-\si_1[\alpha_2(x,y),\alpha_0(z)]- l_0(\alpha_0(y))(\lambda_2(x,z))\\
&~- r_1(\alpha_2(x,z))(\lambda_0(y))- r_1(\alpha_2(x,z))(\beta_0(v))-\mu(\alpha_0(y),\alpha_2(x,z))-\si_1[\alpha_0(y),\alpha_2(x,z)].
\end{align*}
Thus, the two sides are equal if and only if condition (\ref{COC5}) holds.

To derive conditions (\ref{COC6})and (\ref{COC12}). Let $x,y \in \g_0$ and $v \in V_0$. Then

\begin{align*}
 l_0(\alpha_0(x))(v)=&[\sigma_0(\alpha_0(x)),v]=[{F}_0\sigma_0(x)-\lambda_0(x),v]=[{F}_0\sigma_0(x),{F}_0{F}_0^{-1}(v)]\\
=&{F}_0[\si_0(x),{F}_0^{-1}(v)]-\dM_V{F}_2(\si_0(x),{F}_0^{-1}(v))=(\beta_0 l_0(x)\beta_0^{-1})(v).
\end{align*}
And
\begin{align*}
&l_2(\alpha_0(x),\alpha_0(y))(v)+ l_1\alpha_2(x,y)(v)\\
=&-\widehat{l_3}(\si_0\alpha_0(x),\si_0\alpha_0(y),v)+[{F}_2(\si_0(x),\si_0(y))-\lambda_2(x,y),v]\\
=&-\widehat{l_3}({F}_0\sigma_0(x)-\lambda_0(x),{F}_0\sigma_0(y)-\lambda_0(y),v)+[{F}_2(\si_0(x),\si_0(y)),v]\\
=&-\widehat{l_3}({F}_0\sigma_0(x),{F}_0\sigma_0(y),{F}_0{F}_0^{-1}(v))+[{F}_2(\si_0(x),\si_0(y)),{F}_0{F}_0^{-1}(v)]\\
=&-{F}_1(\widehat{l_3}(\sigma_0(x),\sigma_0(y),{F}_0^{-1}(v)))+{F}_2(\si_0(x), [\si_0(y),{F}_0^{-1}(v)])\\
&~- {F}_2([\si_0(x),\si_0(y)],{F}_0^{-1}(v))- {F}_2(\si_0(y),[\si_0(x),{F}_0^{-1}(v)])\\
&~+ [{F}_0\si_0(x), {F}_2(\si_0(y),{F}_0^{-1}(v))] - [{F}_0\si_0(y),{F}_2(\si_0(x),{F}_0^{-1}(v))]\\
=&-\beta_1(\widehat{l_3}(\sigma_0(x),\sigma_0(y),\beta_0^{-1}(v)))\\
=&(\beta_1 l_2(x,y)\beta_0^{-1})(v).
\end{align*}

Similarly, we derive conditions (\ref{COC7})-(\ref{COC11}),  (\ref{COC13}) and (\ref{COC14}).

Conversely, suppose that conditions (\ref{COC1}) -(\ref{COC14}) are given. Then we define ${F}:\hg \to \hg$ by
\begin{align*}
{F}_0(\si_0(x)+v)=&~\si_0\alpha_0(x)+\lambda_0(x)+\beta_0(v).\\
{F}_1(\si_0(a)+m)=&~\si_1\alpha_1(a)+\lambda_1(a)+\beta_1(m).\\
{F}_2(\si_0(x)+u,\si_0(y)+v)=&~\si_0\alpha_2(x,y)+\lambda_2(x,y).
\end{align*}
Since all the maps are linear, it follows that ${F}$ is linear. Clearly, ${F}_0(v)=\beta_0(v)$ for all $v\in V_0$, ${F}_1(m)=\beta_1(m)$ for all $m\in V_1$. Further, we have
$$\overline{{F}}_0(x)=p_0 \big({F}_0 \si_0(x)\big)=p_0\big( \lambda_0(x)+\si_0\alpha_0(x)\big)=p_0\big(\si_0\alpha_0(x)\big)= \alpha_0(x)$$
$$\overline{{F}}_1(a)=p_1 \big({F}_1 \si_1(a)\big)=p_1\big( \lambda_1(a)+\si_1\alpha_1(a)\big)=p_1\big(\si_1\alpha_1(a)\big)= \alpha_1(a)$$
$$\overline{{F}}_2(x,y)=p_1 \big({F}_2 (\si_0(x),\si_0(y))\big)=p_1\big( \lambda_2(x,y)+\si_1\alpha_2(x,y)\big)=p_1\big(\si_1\alpha_2(x,y)\big)= \alpha_2(x,y)$$
for all $x, y \in \fg_0, a \in \fg_1$.

Suppose that ${F}_0(\si_0(x)+v)=0$. This implies that $\si_0\alpha_0(x)=0$. Since $\si_0$ and $\alpha_0$ are both injective, it follows that $x=0$. This further implies that $v=0$, and hence ${F}_0$ is injective. Similarly, ${F}_1$ and ${F}_2$ are also injective.

Let $e_1=\si_0(x)+u$, $e_2=\si_0(y)+v$, $e_3=\si_0(z)+w$ and $h_1=\si_1(a)+m$, for all $x, y, z \in \g_0$, $u,v \in V_0$, $a \in \g_1$ and $ m\in V_1$. We have
\begin{align*}
&{F}_0\widehat{\dM}(h_1)-\widehat{\dM}{F}_1(h_1)\\
=&{F}_0\widehat{\dM}(\sigma_1(a)+m)-\widehat{\dM}{F}_1(\sigma_1(a)+m)\\
=&{F}_0(\psi(a)+\sigma_0\dM_\g(a)+\dM_V(m))-\widehat{\dM}(\lambda_1(a)+\sigma_1(\alpha_1(a))+\beta_1(m))\\
=&\beta_0\psi(a)+\lambda_0(\dM_\g(a))+\beta_0\dM_V(m)+\alpha_0\sigma_0\dM_\g(a)\\
&~-\psi(\alpha_1(a))-\sigma_0\dM_\g\alpha_1(a)-\dM_V(\lambda_1(a))-\dM_V(\beta_1(m))\\
=&~0.\\
\\
&{F}_{0}[e_1,e_2]-[{F}_{0}(e_1),{F}_{0}(e_2)]-\widehat{\dM}{F}_{2}(e_1,e_2)\\
=&{F}_{0}[\sigma_0(x)+u,\sigma_0(y)+v]-[{F}_{0}(\sigma_0(x)+u),{F}_{0}(\sigma_0(y)+v)]-\widehat{\dM}{F}_{2}(\sigma_0(x)+u,\sigma_0(y)+v)\\
=&\beta_0\omega(x,y)+\alpha_0\sigma_0[x,y]+\beta_0 l_0(x)(v)+\beta_0 r_0(y)(u)- l_0(\alpha_0(x))(\lambda_0(y))\\
&~- l_0(\alpha_0(x))(\beta_0(v))- r_0(\alpha_0(y))(\lambda_0(x))- r_0(\alpha_0(y))(\beta_0(u))-\omega(\alpha_0(x),\alpha_0(y))\\
&~-\sigma_0[\alpha_0(x),\alpha_0(y)]-\psi(\alpha_2(x,y))-\sigma_0\dM_\g\alpha_2(x,y)-\dM_V\lambda_2(x,y)\\
=&0.
\end{align*}
Similarly, we have

$${F}_{1}[e_1,h_1]-[{F}_0(e_1),{F}_{1}(h_1)]-{F}_{2}(e_1,\widehat{\dM}(h_1))=0,$$
$${F}_{1}[h_1,e_1]-[{F}_1(h_1),{F}_0(e_1)]-{F}_{2}(\widehat{\dM}(h_1),e_1)=0.$$
Next, we get
\begin{align*}
&{F}_1(\widehat{l_3}(e_1,e_2,e_2))-\widehat{l_3}({F}_0(e_1),{F}_0(e_2),{F}_0(e_3))\\
=&{F}_1\widehat{l_3}(\sigma_0(x),\sigma_0(y),\sigma_0(z))+\beta_1\widehat{l_3}(\sigma_0(x),\sigma_0(y),w)+\beta_1\widehat{l_3}(\sigma_0(x),v,\sigma_0(z))\\
&~+\beta_1\widehat{l_3}(u,\sigma_0(y),\sigma_0(z))-\widehat{l_3}(\sigma_0\alpha_0(x),\sigma_0\alpha_0(y),\sigma_0\alpha_0(z))-\widehat{l_3}(\sigma_0\alpha_0(x),\sigma_0\alpha_0(y),\lambda_0(z))\\
&~-\widehat{l_3}(\sigma_0\alpha_0(x),\sigma_0\alpha_0(y),\beta_0(w))-\widehat{l_3}(\sigma_0\alpha_0(x),\lambda_0(y),\sigma_0\alpha_0(z))-\widehat{l_3}(\sigma_0\alpha_0(x),\beta_0(v),\sigma_0\alpha_0(z))\\
&~-\widehat{l_3}(\lambda_0(x),\sigma_0\alpha_0(y),\sigma_0\alpha_0(z))-\widehat{l_3}(\beta_0(u),\sigma_0\alpha_0(y),\sigma_0\alpha_0(z))\\
=&\sigma_1\alpha_1l^{\g}_3(x,y,z)+\lambda_1(l^{\g}_3(x,y,z))-\sigma_1l^\g_3(\alpha_0(x),\alpha_0(y),\alpha_0(z))+\beta_1\theta(x,y,z)-\theta(\alpha_0(x),\alpha_0(y),\alpha_0(z))\\
&~+ l_2(\alpha_0(x),\alpha_0(y))(\lambda_0(z))+ m_2(\alpha_0(x),\alpha_0(z))(\lambda_0(v))+ r_2(\alpha_0(y),\alpha_0(z))(\lambda_0(x))\\
&~+ l_2(\alpha_0(x),\alpha_0(y))(\beta_0(w))-\beta_1 l_2(x,y)(w)+ m_2(\alpha_0(x),\alpha_0(z))(\beta_0(v))-\beta_1 m_2(x,z)(v)\\
&~+ r_2(\alpha_0(y),\alpha_0(z))(\beta_0(u))-\beta_1 r_2(y,z)(u)\\
=&\lambda_2(x,[y,z])+\sigma_1\alpha_2([x,y],z)-\lambda_2([x,y],z)-\sigma_1\alpha_2(x,[y,z])-\lambda_2(y,[x,z])\\
&~-\sigma_1\alpha_2(y,[x,z])+ l_0(\alpha_0(x))(\lambda_2(y,z))+ r_1(\alpha_2(y,z))(\lambda_0(x))+ r_1(\alpha_2(y,z))(\beta_0(v))\\
&~+\mu(\alpha_0(x),\alpha_2(y,z))+\si_1[\alpha_0(x),\alpha_2(y,z)]- r_0(\alpha_0(z))(\lambda_2(x,y))- l_1(\alpha_2(x,y))(\lambda_0(z))\\
&~- l_1(\alpha_2(x,y))(\beta_0(w))-\nu(\alpha_2(x,y),\alpha_0(z))-\si_1[\alpha_2(x,y),\alpha_0(z)]- l_0(\alpha_0(y))(\lambda_2(x,z))\\
&~- r_1(\alpha_2(x,z))(\lambda_0(y))- r_1(\alpha_2(x,z))(\beta_0(v))-\mu(\alpha_0(y),\alpha_2(x,z))-\si_1[\alpha_0(y),\alpha_2(x,z)]\\
=&{F}_2(\sigma_0(x)+u, \omega(y,z)+\sigma_0[y,z]+ l_0(y)(w)+ r_0(z)(v))\\
&~-{F}_2(\omega(x,y)+\sigma_0[x,y]+ l_0(x)(v)+ r_0(y)(u), \sigma_0(z)+w)\\
&~-{F}_2(\sigma_0(y)+v, \omega(x,z)+\sigma_0[x,z]+ l_0(x)(w)+ r_0(z)(y))\\
&~+[\lambda_0(x)+\sigma_0\alpha_0(x)+\beta_0(u),\lambda_2(y,z)+\sigma_1\alpha_2(y,z)]\\
&~-[\lambda_2(x,y)+\sigma_1\alpha_2(x,y),\lambda_0(z)+\sigma_0\alpha_0(z)+\beta_0(w)]\\
&~-[\lambda_0(y)+\sigma_0\alpha_0(y)+\beta_0(v),\lambda_2(x,z)+\sigma_1\alpha_2(x,z)]\\
=&{F}_2(\sigma_0(x)+u, [\sigma_0(y)+v,\sigma_0(z)+w])- {F}_2([\sigma_0(x)+u,\sigma_0(y)+v],\sigma_0(z)+w)\\
 &~- {F}_2(\sigma_0(y)+v,[\sigma_0(x)+u,\sigma_0(z)+w])
+ [{F}_0(\sigma_0(x)+u), {F}_2(\sigma_0(y)+v,\sigma_0(z)+w)] \\
&~- [{F}_2(\sigma_0(x)+u,\sigma_0(y)+v), {F}_0(\sigma_0(z)+w)] - [{F}_0(\sigma_0(y)+v), {F}_2(\sigma_0(x)+u,\sigma_0(z)+w)]\\
&~-[\lambda_0(y)+\sigma_0\alpha_0(y)+\beta_0(v),\lambda_2(x,z)+\sigma_1\alpha_2(x,z)]\\
=&{F}_2(e_1, [e_2,e_3])- {F}_2([e_1,e_2],e_3) - {F}_2(e_2,[e_1,e_3])+ [{F}_0(e_1), {F}_2(e_2,e_3)] \\
  &- [{F}_2(e_1,e_2), {F}_0(e_3)] - [{F}_0(e_2),{F}_2(e_1,e_3)]
\end{align*}
This proves that ${F} : \widehat{\g} \rightarrow \widehat{\g}$ is a Leibniz 2-algebras homomorphism.

 Let $\sigma(x)+u \in \widehat\g_0,\sigma(a)+m \in \widehat\g_{1}$, Taking $u'=\be_0{^{-1}}\big(u-\lambda_0(\al_0{^{-1}}(x))\big)$,  $m'=\be_1{^{-1}}\big(m-\lambda_1(\al_1{^{-1}}(a))\big)$,
 $x'=\al_0{^{-1}}(x)$, $a'=\al_1{^{-1}}(a)$ yields
\begin{eqnarray*}
F_0\big(\sigma(x')+u'\big) &=&\lambda_0(x')+\sigma(\al_0(x'))+\be_0(u')\\
&=& \lambda_0\big(\al_0^{-1}(x)\big)+\sigma\big(\al_0(\al_0^{-1}(x))\big)
+\be_0\big(\be_0^{-1}\big(u-\lambda_0(\al_0^{-1}(x))\big)\big)\\
&=& \lambda_0\big(\al_0^{-1}(x)\big)+\sigma(x)+u-\lambda_0(\al_0^{-1}(x))\\
&=& \sigma(x)+u,
\end{eqnarray*}
\begin{eqnarray*}
F_1\big(\sigma(a')+m'\big) &=&\lambda_1(a')+\sigma(\al_1(a'))+\be_1(m')\\
&=& \lambda_1\big(\al_1^{-1}(a)\big)+\sigma\big(\al_1(\al_1^{-1}(a))\big)
+\be_1\big(\be_1^{-1}\big(m-\lambda_1(\al_1^{-1}(a))\big)\big)\\
&=& \lambda_1\big(\al_1^{-1}(a)\big)+\sigma(a)+m-\lambda_1(\al_1^{-1}(a))\\
&=& \sigma(a)+m.
\end{eqnarray*}
Hence $F_0,F_1$ is surjective, there by $F_2$ is surjective, then $F=(F_0,F_1,F_2)$ is surjective.
This proves that the pair $(\beta, \alpha) \in \mathrm{Aut} (V) \times \mathrm{Aut} (\mathfrak{g})$ is inducible.
\end{proof}

\begin{rmk}
A pair $(\be,\, \al) \in \Aut(V) \times \Aut(\fg)$ is called compatible if the conditions (\ref{COC6})-(\ref{COC14}) hold.
\end{rmk}

It is easy to see that the set
$$C_\rho(\fg,V)=\{(\beta,\, \alpha) \in \Aut(V) \times \Aut(\fg)|(\beta,\, \alpha) ~\text{satisfies ~  (\ref{COC6})-(\ref{COC14})} \}$$
of all compatible pairs is a subgroup of $\Aut(V) \times \Aut(\fg)$. Conditions (\ref{COC6})-(\ref{COC14}) of the above theorem shows that every inducible pair is compatible.

Next, we show that the obstruction to inducibility of a pair of automorphisms in $\Aut(V) \times \Aut(\fg)$ lies in the second cohomology group $H^2(\fg,V)$.
\begin{lem}
$\Phi$  is a group homomorphism from $\operatorname{Aut}_M(\widehat{\mathfrak{g}}) $ to $ C_\rho(\fg,V) $.
\end{lem}
\begin{proof}
First we prove that the image of  $\Phi$ is contained in $ C_\rho(\fg,V)$.
To derive conditions \ref{COC6} and \ref{COC12}, we use the fact that ${F}$ is an isomorphism. Let $x,y \in \g_0$, $a \in \g_1$, $u,v \in V_0$ and $m\in V_1$.
Since
$$p_0(\si_0\alpha_0(x)-{F}_0 \si_0(x))=\alpha_0(x)-p_0{F}_0 \si_0(x))=0,$$
$$p_1(\si_1\alpha_1(a)-{F}_1 \si_1(a))=\alpha_1(a)-p_1{F}_1 \si_1(a))=0,$$
$$p_1(\si_1\alpha_2(x,y)-{F}_2 (\si_0(x),\si_0(y))=\alpha_2(x,y)-p_1{F}_2 (\si_0(x),\si_0(y))=0,$$
then we have
\begin{align*}
&~\big[\si_0\alpha_0(x),v\big]=\big[{F}_0\si_0(x),v\big],\quad\big[\si_1\alpha_2(x,y),v\big]=\big[{F}_2(\si_0(x),\si_0(y)),v\big].\\
&~\big[\si_1\alpha_1(a),m\big]=\big[{F}_1\si_1(a),m\big],\quad\widehat{l_3}(\si_0\alpha_0(x),\si_0\alpha_0(y),v)=\widehat{l_3}({F}_0\sigma_0(x),{F}_0\sigma_0(y),v).
\end{align*}
Thus
\begin{align*}
 l_0(\alpha_0(x))(v)=&[\sigma_0(\alpha_0(x)),v]=[{F}_0\sigma_0(x),v]=[{F}_0\sigma_0(x),{F}_0{F}_0^{-1}(v)]\\
=&{F}_0[\si_0(x),{F}_0^{-1}(v)]-\dM_V{F}_2(\si_0(x),{F}_0^{-1}(v))=(\beta_0 l_0(x)\beta_0^{-1})(v).
\end{align*}
And
\begin{align*}
& l_2(\alpha_0(x),\alpha_0(y))(v)+ l_1\alpha_2(x,y)(v)\\
=&-\widehat{l_3}(\si_0\alpha_0(x),\si_0\alpha_0(y),v)+[{F}_2(\si_0(x),\si_0(y)),v]\\
=&-\widehat{l_3}({F}_0\sigma_0(x),{F}_0\sigma_0(y),v)+[{F}_2(\si_0(x),\si_0(y)),v]\\
=&-\widehat{l_3}({F}_0\sigma_0(x),{F}_0\sigma_0(y),{F}_0{F}_0^{-1}(v))+[{F}_2(\si_0(x),\si_0(y)),{F}_0{F}_0^{-1}(v)]\\
=&-{F}_1(\widehat{l_3}(\sigma_0(x),\sigma_0(y),{F}_0^{-1}(v)))+{F}_2(\si_0(x), [\si_0(y),{F}_0^{-1}(v)])\\
&~- {F}_2([\si_0(x),\si_0(y)],{F}_0^{-1}(v))- {F}_2(\si_0(y),[\si_0(x),{F}_0^{-1}(v)])\\
&~+ [{F}_0\si_0(x), {F}_2(\si_0(y),{F}_0^{-1}(v))] - [{F}_0\si_0(y),{F}_2(\si_0(x),{F}_0^{-1}(v))]\\
=&-\beta_1(\widehat{l_3}(\sigma_0(x),\sigma_0(y),\beta_0^{-1}(v)))\\
=&(\beta_1 l_2(x,y)\beta_0^{-1})(v).
\end{align*}
Similarly, we derive conditions (\ref{COC7})-(\ref{COC11}) and (\ref{COC13})-(\ref{COC14}).

Second we show that $\Phi$ a group homomorphism
$$\Phi({F}{F}')= \Phi({F})\Phi({F}').$$
In fact,
\begin{eqnarray*}
\Phi({F}_0{F}'_0)(v,x)&=&(\beta_0 \beta'_0,\,\alpha_0 \alpha'_0 )(v,x)\\
& = &(\beta_0,\alpha_0)(\beta'_0,\alpha'_0)(v,x)\\
& = &\Phi({F}_0) (\beta'_0,\alpha'_0)(v,x)\\
& = &\Phi({F}_0)\Phi({F}'_0)(v,x).
\end{eqnarray*}
Similarly, we have
\begin{align*}
\Phi({F}_1{F}'_1)(m,a)=&~\Phi({F}_1)\Phi({F}'_1)(m,a).\\
\Phi({F}_1{F}'_2)(x+u,y+v)=&~\Phi({F}'_2)(\Phi({F}_0)(u,x),\Phi({F}_0)(v,y)).
\end{align*}
\end{proof}

\begin{eqnarray*}
&&\text{Let} ~\ker \Phi=\Aut^{V,\fg}(\hg)= \big\{{F} \in \Aut(\hg)~|~ \Phi({F}_0)=(1_{V_0},1_{\fg_0}), ~\Phi({F}_1)=(1_{V_1},1_{\fg_1}),~\Phi({F}_2)=(0,0)\big\}.
\end{eqnarray*}
and
\begin{eqnarray*}
&&Z^1(\fg,V) 
= \big\{\lambda \in {C}^1(\fg,V)~\big\}
\end{eqnarray*}
let $\lambda=(\lambda_0, \lambda_1, \lambda_2)$ satisfies the following formulas
\begin{align*}
&\dM_\frkh \lambda_1(a)-\lambda_0 \dM_\g (a)=0,\\
& l_0(x)\lambda_0(y)+ r_0(y)\lambda_0(x)-\lambda_0[x,y]_\g+\dM_\frkh\circ \lambda_2(x,y)=0,\\
& l_0(x)\lambda_1(a)+ r_1(a)\lambda_0(x)-\lambda_1[x,a]_\g+\lambda_2(x,\dM_\g (a))=0,\\
& l_1(a)\lambda_1(x)+ r_0(x)\lambda_0(a)-\lambda_1[a,x]_\g+\lambda_2(\dM_\g (a),x)=0,\\
& l_0(x)\lambda_2(y,z)- r_0(z)\lambda_2(x,y)- l_0(y)\lambda_2(x,z)-\lambda_1(l_3^\g(x,y,z)- l_2(x,y)(\lambda_0(z))\\
&- m_2(x,z)(\lambda_0(y))- r_2(y,z)(\lambda_0(x))+\lambda_2(x, [y,z])- \lambda_2([x,y],z) - \lambda_2(y,[x,z])=0.
\end{align*}

\begin{prop}
Let $\mathcal{E}:0 \to V \to\hg \to \fg \to 0$ be an abelian extension of Leibniz 2-algebras. Then ${Z}^1(\fg,V) \cong \Aut^{V,\fg}(\hg)$ as groups.
\end{prop}
\begin{proof}
Define $\psi: {Z}^1(\fg,V) \to \Aut(\hg)$ by $\psi(\lambda)= {F}_\lambda$, where ${F}_\lambda=({F}_{\lambda_0},{F}_{\lambda_1},{F}_{\lambda_2}):\hg \to \hg$ is given by
\begin{align*}
{F}_{\lambda_0}\big(\si_0(x)+v\big)=&~\si_0(x)+\lambda_0(x)+v,\\
{F}_{\lambda_1}\big(\si_1(a)+m\big)=&~\si_1(a)+\lambda_1(a)+m,\\
{F}_{\lambda_2}\big(\si_0(x)+u,~\si_0(y)+v\big)=&~\lambda_2(x,y),
\end{align*}
{for all}  $u,v \in V_0$, $x,y \in \fg_0$, $a\in \g_1$ and $m \in V_1$.

Since $\si$ and $\lambda$ are linear maps, it follows that ${F}_\lambda$ is also linear. Let $e_1=\si_0(x)+u$, $e_2=\si_0(y)+v$, $e_3=\si_0(z)+w$ and $h_1=\si_1(a)+m$, for all $x, y, z \in \g_0$, $u,v \in V_0$, $a \in \g_1$ and $ m\in V_1$. Then
\begin{align*}
&{F}_{\lambda_0}\widehat{\dM}(h_1)-\widehat{\dM}{F}_{\lambda_1}(h_1)\\
=&{F}_{\lambda_0}\widehat{\dM}(\sigma_1(a)+m)-\widehat{\dM}{F}_{\lambda_1}(\sigma_1(a)+m)\\
=&{F}_{\lambda_0}(\psi(a)+\sigma_0\dM_\g(a)+\dM_V(m))-\widehat{\dM}(\lambda_1(a)+\sigma_1(a)+m)\\
=&\psi(a)+\lambda_0(\dM_\g(a))+\sigma_0\dM_\g(a)+\dM_V(m)\\
&~-\dM_V(\lambda_1(a))-\dM_V(m))-\psi(a)-\si_0\dM_\g(a)\\
=&~0,\\
\\
&{F}_{\lambda_0}[e_1,e_2]-[{F}_{\lambda_0}(e_1),{F}_{\lambda_0}(e_2)]-\widehat{\dM}{F}_{\lambda_2}(e_1,e_2)\\
=&{F}_{\lambda_0}[\sigma_0(x)+u,\sigma_0(y)+v]-[{F}_{\lambda_0}(\sigma_0(x)+u),{F}_{\lambda_0}(\sigma_0(y)+v)]-\widehat{\dM}{F}_{\lambda_2}(\sigma_0(x)+u,\sigma_0(y)+v)\\
=&\si_0[x,y]+\omega(x,y)+\lambda_0[x,y]+ l_0(x)(v)+ r_0(y)(u)-\si_0[x,y]-\omega(x,y)- l_0(x)(v)\\
&~- r_0(y)(u)- l_0(x)(\lambda_0(y))- r_0(y)(\lambda_0(x))-\dM_V\lambda_2(x,y)\\
=&0.
\end{align*}
Similarly, we have

$${F}_{\lambda_1}[e_1,h_1]-[{F}_{\lambda_0}(e_1),{F}_{\lambda_1}(h_1)]-{F}_{\lambda_2}(e_1,\widehat{\dM}(h_1))=0,$$
$${F}_{\lambda_1}[h_1,e_1]-[{F}_{\lambda_1}(h_1),{F}_{\lambda_0}(e_1)]-{F}_{\lambda_2}(\widehat{\dM}(h_1),e_1)=0.$$
Next, we have
\begin{align*}
&{F}_{\lambda_1}(\widehat{l_3}(e_1,e_2,e_2))-\widehat{l_3}({F}_{\lambda_0}(e_1),{F}_{\lambda_0}(e_2),{F}_{\lambda_0}(e_3)\\
=&{F}_{\lambda_1}\widehat{l_3}(\sigma_0(x),\sigma_0(y),\sigma_0(z))+\widehat{l_3}(\sigma_0(x),\sigma_0(y),w)+\widehat{l_3}(\sigma_0(x),v,\sigma_0(z))\\
&~+\widehat{l_3}(u,\sigma_0(y),\sigma_0(z))-\widehat{l_3}(\sigma_0(x),\sigma_0(y),\sigma_0(z))-\widehat{l_3}(\sigma_0(x),\sigma_0(y),\lambda_0(z))\\
&~-\widehat{l_3}(\sigma_0(x),\sigma_0(y),w)-\widehat{l_3}(\sigma_0(x),\lambda_0(y),\sigma_0(z))-\widehat{l_3}(\sigma_0(x),v,\sigma_0(z))\\
&~-\widehat{l_3}(\lambda_0(x),\sigma_0(y),\sigma_0(z))-\widehat{l_3}(u,\sigma_0(y),\sigma_0(x))\\
=&\sigma_1l^{\g}_3(x,y,x)+\lambda_0(l^{\g}_3(x,y,x))-\sigma_1l^\g_3(x,y,z)+\theta(x,y,z)-\theta(x,y,z)\\
&~+ l_2(x,y)(\lambda_0(z))+ m_2(x,z)(\lambda_0(y))+ r_2(y,z)(\lambda_0(x))\\
&~+ l_2(x,y)(w)- l_2(x,y)(w)+ m_2(x,z)(v)- m_2(x,z)(v)\\
&~+ r_2(y,z)(u)- r_2(y,z)(u)\\
=&{F}_{\lambda_2}(\sigma_0(x)+u, \omega(y,z)+\sigma_0[y,z]+ l_0(y)(w)+ r_0(z)(v))\\
&~-{F}_{\lambda_2}(\omega(x,y)+\sigma_0[x,y]+ l_0(x)(v)+ r_0(y)(u), \sigma_0(z)+w)\\
&~-{F}_{\lambda_2}(\sigma_0(y)+v, \omega(x,z)+\sigma_0[x,z]+ l_0(x)(w)+ r_0(z)(y))\\
&~+[\lambda_0(x)+\sigma_0(x)+u,\lambda_2(y,z)]-[\lambda_2(x,y),\lambda_0(z)+\sigma_0(z)+w]\\
&~-[\lambda_0(y)+\sigma_0(y)+v,\lambda_2(x,z)]\\
=&{F}_{\lambda_2}(\sigma_0(x)+u, [\sigma_0(y)+v,\sigma_0(z)+w])- {F}_{\lambda_2}([\sigma_0(x)+u,\sigma_0(y)+v],\sigma_0(z)+w)\\
 &~- {F}_{\lambda_2}(\sigma_0(y)+v,[\sigma_0(x)+u,\sigma_0(z)+w])
+ [{F}_{\lambda_0}(\sigma_0(x)+u), {F}_{\lambda_2}(\sigma_0(y)+v,\sigma_0(z)+w)] \\
&~- [{F}_{\lambda_2}(\sigma_0(x)+u,\sigma_0(y)+v), {F}_{\lambda_0}(\sigma_0(z)+w)] - [{F}_{\lambda_0}(\sigma_0(y)+v), {F}_{\lambda_2}(\sigma_0(x)+u,\sigma_0(z)+w)]\\
=&{F}_{\lambda_2}(e_1, [e_2,e_3])- {F}_{\lambda_2}([e_1,e_2],e_3) - {F}_{\lambda_2}(e_2,[e_1,e_3])+ [{F}_{\lambda_0}(e_1), {F}_{\lambda_2}(e_2,e_3)] \\
  &- [{F}_{\lambda_2}(e_1,e_2), {F}_{\lambda_0}(e_3)] - [{F}_{\lambda_0}(e_2),{F}_{\lambda_2}(e_1,e_3)].
\end{align*}
Thus ${F}_\lambda$ is a Leibniz 2-algebras homomorphism of $\hg$. Since $\si=(\si_0,\si_1)$ is injective,
$${F}_{\lambda_0}(\si_0(x)+v)=0, ~~\text{implies that}~~ v=0 ~~\text{and} ~~x=0.$$
$${F}_{\lambda_1}(\si_1(a)+m)=0, ~~\text{implies that}~~ m=0 ~~ \text{and} ~~ a=0.$$
Finally, if $\si_0(x)+v,~ \si_0(y)+v \in \hg_0$, and $\si_1(a)+m \in \hg_1$ then
$${F}_{\lambda_0}(v-\lambda_0(x)+\si_0(x))= \si_0(x)+v.$$
$${F}_{\lambda_1}(m-\lambda_1(a)+\si_1(a))= \si_1(a)+m.$$
Hence ${F}_\lambda$ is a Leibniz 2-algebras automorphism of $\hg$. Clearly, $\Phi ({F}_{\lambda_0})=(1_{V_0},1_{\g_0})$, $\Phi ({F}_{\lambda_1})=(1_{V_1},1_{\g_1})$, $\Phi ({F}_{\lambda_2})=(0,0)$,and hence $\psi({F})={F}_\lambda \in \Aut^{V,\fg}(\hg)$.

Let $\lambda=(\lambda_0,\lambda_1,\lambda_2), \lambda'=(\lambda'_0 ,\lambda'_1 ,\lambda'_1) \in {Z}^1(\fg,V)$. Then
\begin{eqnarray*}
{F}_{\lambda_0+\lambda'_0}\big(\si_0(x)+v\big) & = &\si_0(x)+  \lambda_0(x)+\lambda'_0(x)+v\\
& = & {F}_{\lambda'_0}\big(\si_0(x)+\lambda_0(x)+v\big)\\
& = & {F}_{\lambda'_0}\big({F}_{\lambda_0}(\si_0(x)+v)\big).
\end{eqnarray*}
Similarly, we have
\begin{align*}
{F}_{\lambda_1+\lambda'_1}\big(\si_1(a)+m\big)=&~{F}_{\lambda'_1}\big({F}_{\lambda_1}(\si_1(a)+m)\big),\\
{F}_{\lambda_1+\lambda'_2}\big(\si_0(x)+u,\si_0(y)+v\big)=&~{F}_{\lambda'_2}\big({F}_{\lambda_0}(\si_0(x)+u),{F}_{\lambda_0}(\si_0(y)+v)\big).
\end{align*}

Thus $\psi$ is a group homomorphism. It is easy to see that $\psi$ is injective. Finally, we show that $\psi$ is surjective onto $\Aut^{V,\fg}(\hg)$. Let ${F} \in \Aut^{V,\fg}(\hg)$. Since $\overline{{F}}=(1_{\g_0},1_{\g_1},0)$, we have
 $$p_0 \big({F}_0 \si_0(x)\big)=x=p_0\big(\si_0(x)\big),~p_1\big({F}_1 \si_0(a)\big)=a=p_1\big(\si_1(a)\big),$$
 $$p_1\big({F}_2(\si_0(x),\si_0(y)\big)=0~~\text{for all}~ x,y \in \fg_0,~a \in \g_1.$$
 This implies that
 $${F}_0\big(\si_0(x)\big)= \lambda_0(x)+\si_0(x),~{F}_1\big(\si_1(a)\big)= \lambda_1(a)+\si_1(a),~{F}_2\big(\si_0(x),\si_0(y)\big)= \lambda_2(x,y)$$
 for some element $\lambda_0(x),\lambda_1(a),\lambda_2(x,y) \in V$. We claim that $\lambda \in {Z}^1(\fg,V)$. Since ${F}$ and $\si$ are linear maps, it follows that $\lambda:\fg \to V$ is also linear. Since ${F}$ is a Leibniz 2-algebras homomorphism. But
\begin{eqnarray*}
{F}_0\widehat{\dM}(\si_1(a))  =  \lambda_0(\dM_{\g}(a))+\si_0(\dM_{\g}(a))+\psi(a).\\
\end{eqnarray*}
On the other hand, we have

\begin{eqnarray*}
\widehat{\dM}{F}_1(\si_1(a))  =\dM_V\lambda_1(a)+\si_0(\dM_{\g}(a))+\psi(a).\\
\end{eqnarray*}
Equating these two expressions, we obtain
$$\dM_\frkh \lambda_1(a)-\lambda_0 \dM_\g (a)=0.$$

Then, we have
\begin{align*}
&{F}_{0}[\si_0(x),\si_0(y)]_\g-[{F}_{0}(\si_0(x)),{F}_{0}(\si_0(y))]\\
=&{F}_{0}(\omega(x,y)+\si_0[x,y])-[\lambda_0(x)+\si_0(x),\lambda_0(y)+\si_0(y)]\\
=&\omega(x,y)+\si_0[x,y]+\lambda_0[x,y]- l_0(x)(\lambda_0(y))\\
&- r_0(y)(\lambda_0(x))-\omega(x,y)-\si_0[x,y].
\end{align*}
On the other hand, we have
\begin{align*}
\widehat{\dM}{F}_{2}(\si_0(x),\si_0(y))=\widehat{\dM}\lambda_2(x,y)=\dM_V\lambda_2(x,y).
\end{align*}
Equating these two expressions, we obtain
$$ l_0(x)\lambda_0(y)+ r_0(y)\lambda_0(x)-\lambda_0[x,y]_\g+\dM_\frkh\circ \lambda_2(x,y)=0.$$

Similarly, we obtain
\begin{align*}
& l_0(x)\lambda_1(a)+ r_1(a)\lambda_0(x)-\lambda_1[x,a]_\g+\lambda_2(x,\dM_\g (a))=0,\\
& l_1(a)\lambda_1(x)+ r_0(x)\lambda_0(a)-\lambda_1[a,x]_\g+\lambda_2(\dM_\g (a),x)=0.
\end{align*}

Next, we get
\begin{align*}
&{F}_1(\widehat{l_3}(\si_0(x),\si_0(y),\si_0(z)))-\widehat{l_3}({F}_0(\si_0(x)),{F}_0(\si_0(y)),{F}_0(\si_0(z)))\\
=&{F}_1(\si_1 l_3(x,y,x)+\theta(x,y,z))-\widehat{l_3}(\lambda_0(x)+\si_0(x),\lambda_0(y)+\si_0(y),\lambda_0(z)+\si_0(z))\\
=&\si_1 l_3(x,y,x)+\theta(x,y,z)+\lambda_0(l_3(x,y,x))-\si_1 l_3(x,y,x)-\theta(x,y,z)\\
&+ l_2(x,y)(\lambda_0(z))+ r_2(y,z)(\lambda_0(x))+ m_2(x,z)(\lambda_0(y)).
\end{align*}
On the other hand, we have
\begin{align*}
&{F}_2(\si_0(x), [\si_0(y),\si_0(z)])- {F}_2([\si_0(x),\si_0(y)],\si_0(z))- {F}_2(\si_0(y),[\si_0(x),\si_0(z)])\\
 &+ [{F}_0(\si_0(x)), {F}_2(\si_0(y),\si_0(z))] - [{F}_2(\si_0(x),\si_0(y)), {F}_0(\si_0(z))]- [{F}_0(\si_0(y)),{F}_2(\si_0(x),\si_0(z))]\\
=&\lambda_2(x,[y,z])-\lambda_2([x,y],z)-\lambda_2(y,[x,z])+ l_0(x)(\lambda_2(y,z))- r_0(z)(\lambda_2(x,y))- l_0(y)(\lambda_2(x,z))
\end{align*}
Equating these two expressions, we obtain
\begin{align*}
& l_0(x)\lambda_2(y,z)- r_0(z)\lambda_2(x,y)- l_0(y)\lambda_2(x,z)-\lambda_1(l_3^\g(x,y,z)- l_2(x,y)(\lambda_0(z))\\
&- m_2(x,z)(\lambda_0(y))- r_2(y,z)(\lambda_0(x))+\lambda_2(x, [y,z])- \lambda_2([x,y],z) - \lambda_2(y,[x,z])=0.
\end{align*}
Hence $\lambda=(\lambda_0,\lambda_1,\lambda_2) \in {Z}^1(\fg,V)$. This completes the proof of the proposition.
\end{proof}
For each $(\beta, \alpha) \in \C_\rho(\fg,V)$, we define
$$\Psi: C_\rho(\fg,V)\to \End (Z^2(\fg,V))$$
by
\begin{align*}
\Psi(\beta,\alpha)\psi(a)=\psi_{(\beta, \alpha)}(a)=&~\beta_0\psi(\alpha^{-1}_0(a)),\\
\Psi(\beta,\alpha)\omega(x,y)=\omega_{(\beta, \alpha)}(x,y)=&~\beta_0\omega(\alpha^{-1}_0(x),\alpha^{-1}_0(y)),\\
\Psi(\beta,\alpha)\mu(x,a)=\mu_{(\beta, \alpha)}(x,a)=&~\beta_1\mu(\alpha^{-1}_0(x),\alpha^{-1}_1(a)),\\
\Psi(\beta,\alpha)\nu(a,x)=\nu_{(\beta, \alpha)}(a,x)=&~\beta_0\nu(\alpha^{-1}_1(a),\alpha^{-1}_0(x)),\\
\Psi(\beta,\alpha)\theta(x,y,z)=\theta_{(\beta, \alpha)}(x,y,z)=&~\beta_0\theta(\alpha^{-1}_0(x),\alpha^{-1}_0(y),\alpha^{-1}_0(z)).
\end{align*}
It can be easily seen that $(\psi,\omega,\mu,\nu,\theta)_{(\psi, \phi)}$ is a 2-cocycle if $(\psi,\omega,\mu,\nu,\theta)$ is a 2-cocycle. Since $(\psi,\omega,\mu,\nu,\theta)$ is the 2-cocycle corresponding to the extension $\mathcal{E}$, this defines a map
$$\varpi_{\mathcal{E}}: C_\rho(\fg,V) \to H^2(\fg,V)$$
given by
$$\varpi_{\mathcal{E}} (\beta, \alpha)= [(\psi,\omega,\mu,\nu,\theta)_{(\beta, \alpha)}]-[(\psi,\omega,\mu,\nu,\theta)],$$
where $[(\psi,\omega,\mu,\nu,\theta)_{(\beta, \alpha)}]$ is the cohomology class of $(\psi,\omega,\mu,\nu,\theta)_{(\beta, \alpha)}$.

This $\varpi_{\mathcal{E}}$ is called the {\bf Wells map}.  We show that there is a left action of $C_\rho(\fg,V)$ on $H^2(\fg,V)$ with respect to which $\varpi_{\mathcal{E}}$ is an inner derivation. Hence $\varpi_{\mathcal{E}}=0$ if and only if this action of $C_\rho(\fg,V)$ on $H^2(\fg,V)$ is trivial.

Let $(\beta,\alpha) \in C_\rho(\fg,V)$ and $(\psi,\omega,\mu,\nu,\theta) \in \Z^2(\fg,V)$. For $x,y \in \fg_0,~ a\in \g_1$, define
\begin{align*}
{}^{(\beta, \alpha)}\psi(a)=&~\beta_0\psi(\alpha^{-1}_0(a)),\\
{}^{(\beta, \alpha)}\omega(x,y)=&~\beta_0\omega(\alpha^{-1}_0(x),\alpha^{-1}_0(y)),\\
{}^{(\beta, \alpha)}\mu(x,a)=&~\beta_1\mu(\alpha^{-1}_0(x),\alpha^{-1}_1(a)),\\
{}^{(\beta, \alpha)}\nu(a,x)=&~\beta_0\nu(\alpha^{-1}_1(a),\alpha^{-1}_0(x)),\\
{}^{(\beta, \alpha)}\theta(x,y,z)=&~\beta_0\theta(\alpha^{-1}_0(x),\alpha^{-1}_0(y),\alpha^{-1}_0(z)).
\end{align*}
The compatibility of $(\beta,\alpha)$ implies that ${}^{(\beta,\alpha)}(\psi,\omega,\mu,\nu,\theta) \in \Z^2(\fg,V)$. Additionally, if $(\psi,\omega,\mu,\nu,\theta) \in B^2(\fg,V)$, then there exists some $\lambda=(\lambda_0,\lambda_1,\lambda_2) \in {C}^1(\fg,V)$ such that $(\psi,\omega,\mu,\nu,\theta)= \partial^1(\lambda)$. Given the compatibility of $(\beta,\alpha)$, it follows that ${}^{(\beta,\alpha)}(\psi,\omega,\mu,\nu,\theta)= \partial^1(\beta \lambda \alpha^{-1})$, where $\beta \lambda \alpha^{-1} \in {C}^1(\fg,V)$. Consequently, $[{}^{(\beta,\alpha)}(\psi,\omega,\mu,\nu,\theta)] \in H^2(\fg,V)$. This clearly defines a left action of $C_\rho(\fg,V)$ on $H^2(\fg,V)$, which is represented by
$$^{(\beta,\alpha)}[(\psi,\omega,\mu,\nu,\theta)]= [{}^{(\beta,\alpha)}(\psi,\omega,\mu,\nu,\theta)].$$

\begin{prop}
Let $\mathcal{E}$ be an extension inducing $\rho$. Then $\varpi_{\mathcal{E}}$ is an inner derivation with respect to the action of $C_\rho(\fg,V)$ on $H^2(\fg,V)$.
\end{prop}

\begin{proof}
Let $\mathcal{E}$ be an extension inducing $\rho$ and $\varpi_{\mathcal{E}}: C_\rho(\fg,V) \to H^2(\fg,V)$ be the Wells map. Then for $(\beta, \alpha)$ in $C_\rho(\fg,V)$, we have

\begin{eqnarray*}
(\psi_{(\beta, \alpha)}-\psi)(a)&=&\beta_0\psi(\alpha^{-1}_0(a))-\psi(a)\\
&=&{}^{(\beta, \alpha)}\psi(a)-\psi(a),\\
(\omega_{(\psi, \phi)}-\omega)(x, y) & = & \beta_0 \big( \omega(\alpha^{-1}_0(x), \alpha^{-1}_0(y)) \big)-\omega(x, y)\\
& = & {}^{(\beta,\alpha)}\omega(x, y)-\omega(x, y),\\
(\mu_{(\beta, \alpha)}-\mu)(x,a)& = & \beta_1\mu(\alpha^{-1}_0(x),\alpha^{-1}_1(a))-\mu(x,a)\\
&=&{}^{(\beta, \alpha)}\mu(x,a)-\mu(x,a),\\
(\nu_{(\beta, \alpha)}-\nu)(a,x)&=&\beta_1\nu(\alpha^{-1}_1(a),\alpha^{-1}_0(x))-\nu(a,x)\\
&=&{}^{(\beta, \alpha)}\nu(a,x)-\nu(a,x),\\
(\theta_{(\beta, \alpha)}-\theta)(x,y,z)&=&\beta_1\theta(\alpha^{-1}_0(x),\alpha^{-1}_0(y),\alpha^{-1}_0(z))\\
&=&{}^{(\beta, \alpha)}\theta(x,y,z)-\theta(x,y,z).
\end{eqnarray*}
This implies $\varpi_{\mathcal{E}} (\beta, \alpha)= [(\psi,\omega,\mu,\nu,\theta)_{(\beta, \alpha)}-(\psi,\omega,\mu,\nu,\theta)]= ^{(\beta, \alpha)}[(\psi,\omega,\mu,\nu,\theta)]-[(\psi,\omega,\mu,\nu,\theta)]$. Hence $\varpi_{\mathcal{E}}$ is an inner derivation with respect to the action of $C_\rho(\fg,V)$ on $H^2(\fg,V)$.
\end{proof}

\begin{lem}\label{keylemma2}
Let $(\beta, \alpha) \in C_\rho(\fg,V)$. Then $(\beta, \alpha)$ is inducible  if and only if $(\beta, \alpha)$ is in the kernel of Wells map $\ker \varpi_{\mathcal{E}}$.
\end{lem}
\begin{proof}
Suppose $(\beta, \alpha)$ is inducible, then there exists a ${F}=({F}_0,{F}_1,{F}_2) \in \Aut_V(\hg)$ such that $\Phi({F})=(\beta, \alpha)$. Let $\si=(\si_0,\si_1): \fg \to \hg$ be sections. Then any elements of $\hg$ can be written uniquely as $\si_0(x)+v$, $\si_1(a)+m$ for some $v\in V_0$, $m\in V_1$ and $x \in \fg_0$, $a\in \g_1$. Now,
$$\alpha_0(x)=p_0 {F}_0 \si_0(x),~\alpha_1(a)=p_1 {F}_1 \si_1(a),~\alpha_2(x,y)=p_1 {F}_2 (\si_0(x),\si_0(y)).$$
 This implies
 $$p_0 \si_0 \alpha_0(x)=p_0 {F}_0 \si_0(x),~p_1 \si_1\alpha_1(a)=p_1 {F}_1 \si_1(a),~p_1\si_1\alpha_2(x,y)=p_1 {F}_2 (\si_0(x),\si_0(y)).$$
 And hence we obtain a linear map $\lambda \in \Hom(\fg,V)$ defined by
\begin{align*}
\lambda_0(x)=&~{F}_0 \sigma_0(x)-\sigma_0\alpha_0(x),\\
\lambda_1(a)=&~{F}_1 \sigma_1(a)-\sigma_1\alpha_1(a),\\
\lambda_2(x,y)=&~{F}_2(\sigma_0(x),\sigma_0(y))-\sigma_1\alpha_2(x,y).
\end{align*}
for some $\lambda_0(x) \in V_0$, $\lambda_1(x) \in V_1$,  $\lambda_2(x,y) \in V_1$.
By the main result of last section, we get  conditions (\ref{COC1}) - (\ref{COC5}). Replacing $\al(x_i)$ by $x_i$, we get
\begin{align*}
&\beta_0\psi (\alpha^{-1}_1(a)) - \psi (a) = \dM_{V}(\lambda_1(\alpha^{-1}_1(a)))-\lambda_0 (\dM_{\g}(\alpha^{-1}_1(a))),\\
\\
&\beta_0 \omega (\alpha^{-1}_0(x), \alpha^{-1}_0(y)) - \omega (x, y) =  l_0(x) (\lambda_0(\alpha^{-1}_0(y))) + r_0(y) (\alpha^{-1}_0(x))  - \lambda_0 ([\alpha^{-1}_0(x), \alpha^{-1}_0(y)]) \\
&+\dM_{V}\lambda_2(\alpha^{-1}_0(y),\alpha^{-1}_0(y))+\psi(\alpha_2(\alpha^{-1}_0(y),\alpha^{-1}_0(y))),\\
\\
&\beta_1 \mu (\alpha^{-1}_0(y), \alpha^{-1}_1(a)) - \mu (x, a) =  l_0(x)( \lambda_1(\alpha^{-1}_1(a))) + r_1(a)( \lambda_0(\alpha^{-1}_0(x)) ) - \lambda_1 ([\alpha^{-1}_0(x), \alpha^{-1}_1(a)])\\
&+\lambda_2(\alpha^{-1}_0(x),\dM_{\g}(\alpha^{-1}_1(a))),\\
\\
&\beta_1 \nu (\alpha^{-1}_1(a), \alpha^{-1}_0(x)) - \nu (a, x) =  r_0(x)( \lambda_1(\alpha^{-1}_1(a))) + l_1(a)( \lambda_0(\alpha^{-1}_0(x)) ) - \lambda_1 ([\alpha^{-1}_1(a), \alpha^{-1}_0(x)])\\
&+\lambda_2(\dM_{\g}(\alpha^{-1}_1(a)),\alpha^{-1}_0(x)),\\
\\
&\beta_1\theta(\alpha^{-1}_0(x),\alpha^{-1}_0(y),\alpha^{-1}_0(z))-\theta(x,y,z)= l_2(x,y)(\lambda_0(\alpha^{-1}_0(z)))- m_2(x,z)(\lambda_0(\alpha^{-1}_0(y)))\\
&- r_2(y,z)(\lambda_0(\alpha^{-1}_0(x)))+\lambda_2(\alpha^{-1}_0(x),[\alpha^{-1}_0(y),\alpha^{-1}_0(z)])-\lambda_2([\alpha^{-1}_0(x),\alpha^{-1}_0(y)],\alpha^{-1}_0(z))\\
&-\lambda_2(\alpha^{-1}_0(y),[\alpha^{-1}_0(x),\alpha^{-1}_0(z)])+ l_0(x)(\lambda_2(\alpha^{-1}_0(y),\alpha^{-1}_0(z)))+ r_1(\alpha_2(\alpha^{-1}_0(y),\alpha^{-1}_0(z)))(\lambda_0(\alpha^{-1}_0(x)))\\
&- r_0(z)(\lambda_2(\alpha^{-1}_0(x),\alpha^{-1}_0(y)))
- l_1(\alpha_2(\alpha^{-1}_0(x),\alpha^{-1}_0(y)))(\lambda_0(\alpha^{-1}_0(z)))- l_0(y)(\lambda_2(\alpha^{-1}_0(x),\alpha^{-1}_0(z)))\\
&- r_1(\alpha_2(\alpha^{-1}_0(x),\alpha^{-1}_0(z)))(\lambda_0(\alpha^{-1}_0(y)))+\mu(x,\alpha_2(\alpha^{-1}_0(y),\alpha^{-1}_0(z)))-\nu(\alpha_2(\alpha^{-1}_0(x),\alpha^{-1}_0(y)),z)\\
&-\mu(y,\alpha_2(\alpha^{-1}_0(x),\alpha^{-1}_0(z))).
\end{align*}
Thus $(\psi,\omega,\mu,\nu,\theta)_{(\beta,\alpha)}-(\psi,\omega,\mu,\nu,\theta)=\partial^1 (\lambda\alpha^{-1})$ and
$$\varpi_{\mathcal{E}} (\beta,\alpha)= [(\psi,\omega,\mu,\nu,\theta)_{(\beta,\alpha)}-\psi,\omega,\mu,\nu,\theta)]=0.$$
Then $(\beta,\alpha)$ is in the kernel of Wells map $\ker \varpi_{\mathcal{E}}$.
\end{proof}

\begin{lem}
The image of  $\Phi$ is equal to the kernel of  the Wells map $\varpi_{\mathcal{E}}$, i.e. $\operatorname{im}\Phi=\operatorname{ker} \varpi_{\mathcal{E}}$.
\end{lem}

\begin{proof}
One note that $ (\beta,\alpha) \in \operatorname{im} \Phi$ if and only if $ (\beta,\alpha)$ is inducible.
Then by the above Lemma \eqref{keylemma2} we obtain the result.
\end{proof}

By Theorem \ref{main-thm2}, the pair $(\beta,\alpha) \in C_\rho(\fg,V)$ is inducible if and only if $(\psi,\omega,\mu,\nu,\theta)_{(\beta,\alpha)}$ is a 2-coboundary. Thus $[(\psi,\omega,\mu,\nu,\theta)_{(\beta,\alpha)}] \in H^2(\fg,V)$ is an obstruction to inducibility of the pair $(\beta,\alpha) \in C_\rho(\fg,V)$. With the preceding discussion, we have derived the following exact sequence of groups associated to an abelian extension of Leibniz 2-algebras and proving Theorem \ref{wells-sequence2}.

\begin{thm} \label{wells-sequence2}
Let $\mathcal{E}:0 \to V \to\hg \to \fg \to 0$ be an abelian extension of Leibniz 2-algebras. Let $\rho: \fg \to  \End(V)$ denote the induced $\fg$-module structure on $V$. Then there is an exact sequence
\begin{equation}\label{wells-seq}
0 \to Z^1(\fg,V) \stackrel{i}{\longrightarrow} \Aut_V(\hg) \stackrel{\Phi}{\longrightarrow} C_\rho(\fg,V) \stackrel{\varpi_{\mathcal{E}}}{\longrightarrow} H^2(\fg,V).
\end{equation}
\end{thm}

\section{Wells exact sequence for  derivations}
In this section, we study
the inducibility of a pair of derivations about abelian extensions of Leibniz 2-algebras.  Theorem \ref{main-thm1} provides
necessary and sufficient conditions for a pair of derivations to be inducible.

The groups of all Leibniz 2-algebra derivations of degree $0$ of $V$, $\widehat{\g}$, and $\fg$ are denoted as $\Der(V)$, $\Der(\widehat{\g})$, and $\Der(\fg)$, respectively. Let $\Der_V(\widehat{\g})$ denote the group of all Leibniz 2-algebra derivations of $\widehat{\g}$ which keep $V$ invariant as a set. Note that a derivation ${D}=({D}_0, {D}_1, {D}_2) \in \Der_V(\widehat{\g})$ induces derivations ${D}|_V=(\beta_0, \beta_1) \in \Der(V)$ given by
 $$\beta_0(v)={D}_0|_V(v),\quad \beta_1(m)={D}_1|_V(m),$$
 for all $u, v \in V_0, m \in V_1$.
 And we can also define a map $\overline{{D}}=(\alpha_0, \alpha_1, \alpha_2) : \mathfrak{g} \rightarrow \mathfrak{g}$ by
\begin{align*}
\alpha_0(x)=p_0{D}_0\sigma_0(x),\quad \alpha_1(a)=p_1{D}_1\sigma_1(a), \quad \alpha_2(x, y)=p_1{D}_2(\sigma_0(x),\sigma_0(y)),
\end{align*}
for all $x, y \in \g_0, a \in \g_1$.

\begin{prop}
Let  ${D}=({D}_0, {D}_1, {D}_2)  \in \Der_V(\widehat{\g})$  be a Leibniz 2-algebra derivation. Then   $\overline{{D}}=(\alpha_0, \alpha_1, \alpha_2)$ defined above is  a Leibniz 2-algebra derivation of $\g$.
\end{prop}
\begin{proof}
${F}$ is a derivation on $\widehat{\g}$ preserving the space $V$, it also preserves the subspace $\mathfrak{g}$, it is easy to see that $\al_0$ and $\al_1$ are bijective maps on $\fg_0$ and $\fg_1$.

For any $x,y \in \g_0, u, v \in \g_1$, we have
\begin{align*}
(\alpha_0\dM_{\g}-\dM_{\g}\alpha_1)(a)=&(p_0{D}_0\sigma_0\dM_{\g}-\dM_{\g}p_1{D}_1\sigma_1)(a)\\
=&(p_0{D}_0\sigma_0\dM_{\g}-p_0{D}_0\widehat{\dM}\sigma_1)(a)\\
=&p_0{D}_0(\sigma_0\dM_{\g}-\widehat{\dM}\sigma_1)(a)\\
=&0,\\
\dM_{\g}\alpha_2(x,y)=&\dM_{\g}p_1{D}_2(\sigma_0(x), \sigma_0(y))\\
=&p_0\widehat{\dM}{D}_2(\sigma_0(x), \sigma_0(y))\\
=&p_0\big(\alpha_0([\sigma_0(x), \sigma_0(y)])-[{D}_0(\sigma_0(x)), \si_0(y)]-[\si_0(x), {D}_0(\sigma_0(y))]\big)\\
=&p_0\alpha_0(\omega(x, y)+ \sigma_0([x,y]))-[\alpha_0(x),y]-[x, \alpha_0(y)]\\
=&\alpha_0([x,y])-[\alpha_0(x),y]-[x, \alpha_0(y)],\\
\end{align*}
Similarly, we have
\begin{align*}
\alpha_2(x, \dM_{\g}(a))=&\alpha_1[x,a]-[\alpha_0(x),a]-[x,\alpha_1(a)],\\
\alpha_2(\dM_{\g}(a),x)=&\alpha_1[a, x]-[\alpha_1(a),x]-[a, \alpha_0(x)],
\end{align*}
Next, we have
\begin{align*}
&\alpha_1(l_3^\g(x,y,z))-l_3^\g(\alpha_0(x),y,z)-l_3^\g(x,\alpha_0(y),z)-l_3^\g(x,y,\alpha_0(z))\\
=&\alpha_1(\widehat{l}_3(\sigma_0(x),\sigma_0(y),\sigma_0(z))-\theta(x,y,z))-l_3^\g(p_0{D}_0\sigma_0(x),y,z)\\
&~-l_3^\g(x,p_0{D}_0\sigma_0(y),z)-l_3^\g(x,y,p_0{D}_0\sigma_0(z))\\
=&p_1\big({D}_1\widehat{l}_3(\sigma_0(x),\sigma_0(y),\sigma_0(z))-\widehat{l}_3({D}_0\sigma_0(x),\si_0(y),\si_0(z))\\
&~-\widehat{l}_3(\si_0(x),{D}_0\sigma_0(y),\si_0(z))-\widehat{l}_3(\si_0(x),\si_0(y),{D}_0\sigma_0(z))\big)\\
=&p_1{D}_2(\sigma_0(x), [\sigma_0(y),\sigma_0(z)])- p_1{D}_2([\sigma_0(x),\sigma_0(y)],\sigma_0(z))\\
 &~- p_1{D}_2(\sigma_0(y),[\sigma_0(x),\sigma_0(z)])+ p_1[{D}_0(\sigma_0(x)), {D}_2(\sigma_0(y),\sigma_0(z))]\\
  &~- p_1[{D}_2(\sigma_0(x),\sigma_0(y)), {D}_0(\sigma_0(z))] - p_1[{D}_0(\sigma_0(y)), {D}_2(\sigma_0(x),\sigma_0(z))]\\
 =&\alpha_2(x, [y,z])- \alpha_2([x,y],z) - \alpha_2(y,[x,z])\\
 &~+[\alpha_0(x), \alpha_2(y,z)] - [\alpha_2(x,y), \alpha_0(z)] - [\alpha_0(y), \alpha_2(x,z)].
\end{align*}
This shows that $\overline{{D}} : \mathfrak{g} \rightarrow \mathfrak{g}$ is a Leibniz 2-algebra derivation.
\end{proof}

By the above discussion, we obtain a group homomorphism
$$\Phi: \Der_V(\widehat{\g}) \to \Der(V) \times \Der(\fg)$$
given by
$$\Phi({D})= ({D}|_V, \overline{{D}}).$$

A pair of derivations $(\beta, \alpha) \in \Der(V)\times \Der(\fg)$ is called  \textbf{inducible} if there exists a ${D} \in \Der_V(\widehat{\g})$ such that
$$\Phi({D}_0)= (\beta_0, \alpha_0),\quad \Phi({D}_1)= (\beta_1, \alpha_1),\quad \Phi({D}_2)= (0, \alpha_2).$$

We give a characterization of when a pair of derivations is inducible.
The proof is similar as the proof of Theorem  \ref{main-thm2} in last section. Thus we omit the details.
\begin{thm}\label{main-thm1}
Let $0 \to V \to \hg\to \fg \to 0$ be an abelian extension of Leibniz 2-algebras   and $(\theta, \phi) \in \Der(V)\times \Der(\fg)$. Then the pair $(\theta, \phi)$ is inducible if and only if exists a linear map $\lambda \in \mathrm{Hom}(\mathfrak{g}, V)$ satisfying the followings
{\footnotesize
\begin{align}
&\label{COCY1}\beta_0(\psi (a)) - \psi (\alpha_1(a)) = \dM_{V}(\lambda_1(a))-\lambda_0 (\dM_{\g}(a)),\\
&\label{COCY2}\beta_0 (\omega (x, y)) - \omega (\alpha_0(x),y)-\omega(x, \alpha_0(y)) = l_0(x) (\lambda_0(y)) + r_0(y) (\lambda_0(x))  - \lambda_0 [x, y]+\dM_{V}\lambda_2(x,y)\nonumber\\
& +\psi(\alpha_2(x,y)),\\
&\label{COCY3}\beta_1 (\mu (x, a)) - \mu (\alpha_0(x), a) - \mu (x, \alpha_1(a))= l_0(x)( \lambda_1(a)) + r_1(a)( \lambda_0(x) ) - \lambda_1 [x, a]+\lambda_2(x,\dM_{\g}(a)),\\
&\label{COCY4}\beta_1 (\nu (a, x)) - \nu (\alpha_1(a), x)- \nu (a, \alpha_0(x)) = r_0(x)( \lambda_1(a)) + l_1(a)( \lambda_0(x) ) - \lambda_1 [a, x]+\lambda_2(\dM_{\g}(a),x),\\
&\label{COCY5}\beta_1\theta(x,y,z)-\theta(\alpha_0(x),y,z)-\theta(x,\alpha_0(y),z)-\theta(x,y,\alpha_0(z))
-\mu(x,\alpha_2(y,z))+l_0(y)(\lambda_2(x,z))\nonumber\\
&+\nu(\alpha_2(x,y),z)+\mu(y,\alpha_2(x,z))
+r_2(y,z)(\lambda_0(x))+m_2(x,z)(\lambda_0(y))
+l_2(x,y)(\lambda_0(z))\nonumber\\
&=\lambda_2(x,[y,z])-\lambda_2([x,y],z)-\lambda_2(y,[x,z])
+l_0(x)(\lambda_2(y,z))
-r_0(z)(\lambda_2(x,y))-\lambda_1(l^{\g}_3(x,y,z))\\
&\label{COCY6}\beta_0l_0(x)(v)-l_0(\alpha_0(x))(v)
-l_0(x)(\beta_0(v))=0,\\
&\label{COCY7}\beta_0r_0(x)(v)
-r_0(\alpha_0(x))(v)-r_0(x)(\beta_0(v))=0,\\
&\label{COCY8}\beta_1l_0(x)(m)-l_0(\alpha_0(x))(m)
-l_0(x)(\beta_1(m))=0,\\
&\label{COCY9}\beta_1r_0(x)(m)
-r_0(\alpha_0(x))(m)-r_0(x)(\beta_1(m))=0,\\
&\label{COCY10}\beta_1l_1(a)(v)-l_1(\alpha_1(a))(v)
-l_1(a)(\beta_0(v))=0,\\
&\label{COCY11}\beta_1r_1(a)(v)
-r_1(\alpha_1(a))(v)-r_1(a)(\beta_0(v))=0,\\
 &\label{COCY12}l_2(\alpha_0(x),y)(w)+l_2(x,\alpha_0(y))(w)+l_2(x,y)(\beta_0(w))+l_1\alpha_2(x,y){\beta_0}(w)=0,\\
 &\label{COCY13}r_2(\alpha_0(x),y)(w)+r_2(x,\alpha_0(y))(w)+r_2(x,y)(\beta_0(w))-r_1\alpha_2(x,y){\beta_0}(w)=0,\\
 &\label{COCY14}m_2(\alpha_0(x),y)(w)+m_2(x,\alpha_0(y))(w)+m_2(x,y)(\beta_0(w))+r_1\alpha_2(x,y){\beta_0}(w)=0.
\end{align}
}
\end{thm}

\begin{rmk}
 A pair $(\be,\, \al) \in \Der(V) \times \Der(\fg)$ is called compatible if the conditions (\ref{COCY6})-(\ref{COCY14}) hold.
It is easy to see that the set
$$ {D}_\rho(\fg,V)=\{(\beta,\, \alpha) \in \Der(V) \times \Der(\fg)|(\beta,\, \alpha) ~\text{satisfies  (\ref{COCY6})-(\ref{COCY14}) } \}.$$
 The $ {D}_\rho(\fg,V)$ is a subgroup of $\Der(V) \times \Der(\fg)$. Conditions (\ref{COCY6})-(\ref{COCY14}) of the above theorem shows that every inducible pair is compatible.
\end{rmk}

Next, we show that the obstruction to inducibility of a pair of derivations in $\Der(V) \times \Der(\fg)$ lies in the Leibniz 2-algebras cohomology $H^2(\fg,V)$.
\begin{lem}
$\Phi$  is a group homomorphism from $\operatorname{Der}_V(\widehat{\mathfrak{g}}) $ to $  {D}_\rho(\fg,V) $.
\end{lem}

\begin{eqnarray*}
&&\text{Let} ~ \ker \Phi=\Der^{V,\fg}(\hg)= \big\{{D} \in \Der(\hg)~|~ \Phi({D}_0)=(1_{V_0},1_{\fg_0}), ~\Phi({D}_1)=(1_{V_1},1_{\fg_1}),~\Phi({D}_2)=(0,0)\big\}.
\end{eqnarray*}
and ${Z}^1(\fg,V) $ are elements $\lambda \in {C}^1(\fg,V)$ such that  the following formulas are satisfied:
\begin{align*}
&\dM_\frkh \lambda_1(a)-\lambda_0 \dM_\g (a)=0,\\
&l_0(x)\lambda_0(y)+r_0(y)\lambda_0(x)-\lambda_0[x,y]_\g+\dM_\frkh\circ \lambda_2(x,y)=0,\\
&l_0(x)\lambda_1(a)+r_1(a)\lambda_0(x)-\lambda_1[x,a]_\g+\lambda_2(x,\dM_\g (a))=0,\\
&l_1(a)\lambda_1(x)+r_0(x)\lambda_0(a)-\lambda_1[a,x]_\g+\lambda_2(\dM_\g (a),x)=0,\\
&l_0(x)\lambda_2(y,z)-r_0(z)\lambda_2(x,y)-l_0(y)\lambda_2(x,z)-\lambda_1(l_3^\g(x,y,z)-\rho^L_{2}(x,y)(\lambda_0(z))\\
&-\rho_2^M(x,z)(\lambda_0(y))-\rho^R_{2}(y,z)(\lambda_0(x))+\lambda_2(x, [y,z])- \lambda_2([x,y],z) - \lambda_2(y,[x,z])=0.
\end{align*}

\begin{prop}
Let $\mathcal{E}:0 \to V \to\hg \to \fg \to 0$ be an abelian extension of Leibniz 2-algebras. Then ${Z}^1(\fg,V) \cong \Der^{V,\fg}(\hg)$ as groups.
\end{prop}
\begin{proof}
Define $\psi: {Z}^1(\fg,V) \to \Der(\hg)$ by $\psi(\lambda)= {D}_\lambda$, where ${D}_\lambda=({D}_{\lambda_0},{D}_{\lambda_1},{D}_{\lambda_2}):\hg \to \hg$ is given by
\begin{align*}
{D}_{\lambda_0}\big(\si_0(x)+v\big)=&~\lambda_0(x),\\
{D}_{\lambda_1}\big(\si_1(a)+m\big)=&~\lambda_1(a),\\
{D}_{\lambda_2}\big(\si_0(x)+u,~\si_0(y)+v\big)=&~\lambda_2(x,y),
\end{align*}
{for all}  $u,v \in V_0$, $x,y \in \fg_0$, $a\in \g_1$ and $m \in V_1$.

Since $\si$ and $\lambda$ are linear maps, it follows that ${D}_\lambda$ is also linear. Let $e_1=\si_0(x)+u$, $e_2=\si_0(y)+v$, $e_3=\si_0(z)+w$ and $h_1=\si_1(a)+m$, for all $x, y, z \in \g_0$, $u,v \in V_0$, $a \in \g_1$ and $ m\in V_1$. Then
\begin{align*}
&{D}_{\lambda_0}\widehat{\dM}(h_1)-\widehat{\dM}{D}_{\lambda_1}(h_1)\\
=&{D}_{\lambda_0}\widehat{\dM}(\sigma_1(a)+m)-\widehat{\dM}{D}_{\lambda_1}(\sigma_1(a)+m)\\
=&{D}_{\lambda_0}(\psi(a)+\sigma_0\dM_\g(a)+\dM_V(m))-\widehat{\dM}(\lambda_1(a))\\
=&\lambda_0(\dM_\g(a))-\dM_V(\lambda_1(a))\\
=&~0,\\
\\
&{D}_{\lambda_0}[e_1,e_2]-[{D}_{\lambda_0}(e_1),e_2]-[e_1,{D}_{\lambda_0}(e_2)]-\widehat{\dM}{D}_{\lambda_2}(e_1,e_2)\\
=&{D}_{\lambda_0}[\sigma_0(x)+u,\sigma_0(y)+v]-[{D}_{\lambda_0}(\sigma_0(x)+u),\sigma_0(y)+v]\\
&-[\sigma_0(x)+u,{D}_{\lambda_0}(\sigma_0(y)+v)]-\widehat{\dM}{D}_{\lambda_2}(\sigma_0(x)+u,\sigma_0(y)+v)\\
=&\lambda_0[x,y]-r_0(y)(\lambda_0(x))-l_0(x)(\lambda_0(y))-\dM_V\lambda_2(x,y)\\
=&0.
\end{align*}
Similarly, we have

$${D}_{\lambda_1}[e_1,h_1]-[{D}_{\lambda_0}(e_1),h_1]-[e_1,{D}_{\lambda_1}(h_1)]-{D}_{\lambda_2}(e_1,\widehat{\dM}(h_1))=0,$$
$${D}_{\lambda_1}[h_1,e_1]-[{D}_{\lambda_1}(h_1),e_1]-[h_1,{D}_{\lambda_0}(e_1)]-{D}_{\lambda_2}(\widehat{\dM}(h_1),e_1)=0.$$
Next
\begin{align*}
&{D}_{\lambda_1}(\widehat{l_3}(e_1,e_2,e_2))-\widehat{l_3}({D}_{\lambda_0}(e_1),e_2,e_3)-\widehat{l_3}(e_1,{D}_{\lambda_0}(e_2),e_3)-\widehat{l_3}(e_1,e_2,{D}_{\lambda_0}(e_3))\\
=&{D}_{\lambda_1}(\widehat{l_3}(\sigma_0(x)+u,\sigma_0(y)+v,\sigma_0(z)+w))-\widehat{l_3}({D}_{\lambda_0}(\sigma_0(x)+u),\sigma_0(y)+v,\sigma_0(z)+w)\\
&-\widehat{l_3}(\sigma_0(x)+u,{D}_{\lambda_0}(\sigma_0(y)+v),\sigma_0(z)+w)-\widehat{l_3}(\sigma_0(x)+u,\sigma_0(y)+v,{D}_{\lambda_0}(\sigma_0(z)+w))\\
=&\lambda_0(l^{\g}_3(x,y,x))+r_2(y,z)(\lambda_0(x))+m_2(x,z)(\lambda_0(y))+l_2(x,y)(\lambda_0(z))\\
=&\lambda_2(x,[y,z])-\lambda_2([x,y],z)-\lambda_2(y,[x,z])\\
&~+l_0(x)(\lambda_2(y,z))-r_0(z)(\lambda_2(x,y))-l_0(y)(\lambda_2(x,z))\\
=&{D}_{\lambda_2}(\sigma_0(x)+u, \omega(y,z)+\sigma_0[y,z]+l_0(y)(w)+r_0(z)(v))\\
&~-{D}_{\lambda_2}(\omega(x,y)+\sigma_0[x,y]+l_0(x)(v)+r_0(y)(u), \sigma_0(z)+w)\\
&~-{D}_{\lambda_2}(\sigma_0(y)+v, \omega(x,z)+\sigma_0[x,z]+l_0(x)(w)+r_0(z)(y))\\
&~+[\sigma_0(x)+u,\lambda_2(y,z)]-[\lambda_2(x,y),\sigma_0(z)+w]-[\sigma_0(y)+v,\lambda_2(x,z)]\\
=&{D}_{\lambda_2}(\sigma_0(x)+u, [\sigma_0(y)+v,\sigma_0(z)+w])- {D}_{\lambda_2}([\sigma_0(x)+u,\sigma_0(y)+v],\sigma_0(z)+w)\\
 &~- {D}_{\lambda_2}(\sigma_0(y)+v,[\sigma_0(x)+u,\sigma_0(z)+w])
+ [\sigma_0(x)+u, {D}_{\lambda_2}(\sigma_0(y)+v,\sigma_0(z)+w)] \\
&~- [{D}_{\lambda_2}(\sigma_0(x)+u,\sigma_0(y)+v),\sigma_0(z)+w] - [\sigma_0(y)+v, {D}_{\lambda_2}(\sigma_0(x)+u,\sigma_0(z)+w)]\\
=&{D}_{\lambda_2}(e_1, [e_2,e_3])- {D}_{\lambda_2}([e_1,e_2],e_3) - {D}_{\lambda_2}(e_2,[e_1,e_3])+ [e_1, {D}_{\lambda_2}(e_2,e_3)] \\
  &- [{D}_{\lambda_2}(e_1,e_2), e_3] - [e_2,{D}_{\lambda_2}(e_1,e_3)]
\end{align*}
Thus ${D}_\lambda$ is a derivation of Leibniz 2-algebras of $\hg$.

Clearly, $\Phi ({D}_{\lambda_0})=(1_{V_0},1_{\g_0})$, $\Phi ({D}_{\lambda_1})=(1_{V_1},1_{\g_1})$, $\Phi ({D}_{\lambda_2})=(0,0)$, and hence $\psi({D})={D}_\lambda \in \Der^{V,\fg}(\hg)$.
Let $\lambda=(\lambda_0,\lambda_1,\lambda_2), \lambda'=(\lambda'_0 ,\lambda'_1 ,\lambda'_1) \in {Z}^1(\fg,V)$. Then
\begin{eqnarray*}
{D}_{\lambda_0+\lambda'_0}\big(\si_0(x)+v\big) & = &\lambda_0(x)+\lambda'_0(x)\\
& = & {D}_{\lambda'_0}\big(\si_0(x)+\lambda_0(x)\big)\\
& = & {D}_{\lambda'_0}\big({D}_{\lambda_0}(\si_0(x)+v)\big).
\end{eqnarray*}
Similarly, we have
\begin{align*}
{D}_{\lambda_1+\lambda'_1}\big(\si_1(a)+m\big)=&~{D}_{\lambda'_1}\big({D}_{\lambda_1}(\si_1(a)+m)\big).\\
{D}_{\lambda_1+\lambda'_2}\big(\si_0(x)+u,\si_0(y)+v\big)=&~{D}_{\lambda'_2}\big({D}_{\lambda_0}(\si_0(x)+u),{D}_{\lambda_0}(\si_0(y)+v)\big).
\end{align*}

Thus $\psi$ is a group homomorphism. It is easy to see that $\psi$ is injective. Finally, we show that $\psi$ is surjective onto $\Der^{V,\fg}(\hg)$.

Let ${D} \in \Der^{V,\fg}(\hg)$. Since $\overline{{D}}=(1_{\g_0},1_{\g_1},0)$, we have
 $$p_0 \big({D}_0 \si_0(x)\big)=x=p_0\big(\si_0(x)\big),~p_1\big({D}_1 \si_0(a)\big)=a=p_1\big(\si_1(a)\big),$$
 $$p_1\big({D}_2(\si_0(x),\si_0(y)\big)=0 ~~\text{for all}~ x,y \in \fg_0,~a \in \g_1.$$
 This implies that
 $${D}_0\big(\si_0(x)\big)= \lambda_0(x)+\si_0(x),~{D}_1\big(\si_1(a)\big)= \lambda_1(a)+\si_1(a),~{D}_2\big(\si_0(x),\si_0(y)\big)= \lambda_2(x,y)$$
 for some element $\lambda_0(x),\lambda_1(a),\lambda_2(x,y) \in V$. We claim that $\lambda \in {Z}^1(\fg,V)$. Since ${D}$ and $\si$ are linear maps, it follows that $\lambda:\fg \to V$ is also linear. Since ${D}$ is a derivation of Leibniz 2-algebras, we have
\begin{eqnarray*}
{D}_0\widehat{\dM}(\si_1(a)) & = & \lambda_0(\dM_{\g}(a))\\.
\end{eqnarray*}
On the other hand, we have

\begin{eqnarray*}
\widehat{\dM}{D}_1(\si_1(a))  =\dM_V\lambda_1(a)\\.
\end{eqnarray*}
Equating these two expressions, we obtain
$$\dM_\frkh \lambda_1(a)-\lambda_0 \dM_\g (a)=0.$$

Then, we have
\begin{align*}
&{D}_{0}[\si_0(x),\si_0(y)]_\g-[{D}_{0}(\si_0(x)),\si_0(y)]-[\si_0(x),{D}_{0}(\si_0(y))]\\
=&{D}_{0}(\omega(x,y)+\si_0[x,y])-[\lambda_0{(x)},\si_0(y)]-[\si_0(x),\lambda_0{(y)}]\\
=&\lambda_0[x,y]-l_0(x)(\lambda_0(y))-r_0(y)(\lambda_0(x)).
\end{align*}
On the other hand, we have
\begin{align*}
\widehat{\dM}{D}_{2}(\si_0(x),\si_0(y))=\widehat{\dM}\lambda_2(x,y)=\dM_V\lambda_2(x,y).
\end{align*}
Equating these two expressions, we obtain
$$l_0(x)\lambda_0(y)+r_0(y)\lambda_0(x)-\lambda_0[x,y]_\g+\dM_\frkh\circ \lambda_2(x,y)=0.$$

Similarly, we obtain
\begin{align*}
&l_0(x)\lambda_1(a)+r_1(a)\lambda_0(x)-\lambda_1[x,a]_\g+\lambda_2(x,\dM_\g (a))=0,\\
&l_1(a)\lambda_1(x)+r_0(x)\lambda_0(a)-\lambda_1[a,x]_\g+\lambda_2(\dM_\g (a),x)=0.
\end{align*}

Next, we get
\begin{align*}
&{D}_1(\widehat{l_3}(\si_0(x),\si_0(y),\si_0(z)))-\widehat{l_3}({D}_0(\si_0(x)),\si_0(y),\si_0(z))\\
&-\widehat{l_3}(\si_0(x),{D}_0(\si_0(y)),\si_0(z))-\widehat{l_3}(\si_0(x),\si_0(y),{D}_0(\si_0(z)))\\
=&{D}_1(\si_1 l_3(x,y,x)+\theta(x,y,z))-\widehat{l_3}(\lambda_0(x),\si_0(y),\si_0(z))\\
&-\widehat{l_3}(\si_0(x),\lambda_0(y),\si_0(z))-\widehat{l_3}(\si_0(x),\si_0(y),\lambda_0(z))\\
=&\lambda_0(l_3(x,y,x))+l_2(x,y)(\lambda_0(z))+r_2(y,z)(\lambda_0(x))+m_2(x,z)(\lambda_0(y)).
\end{align*}
On the other hand, we have
\begin{align*}
&{D}_2(\si_0(x), [\si_0(y),\si_0(z)])- {D}_2([\si_0(x),\si_0(y)],\si_0(z))- {D}_2(\si_0(y),[\si_0(x),\si_0(z)])\\
 &+ [\si_0(x), {D}_2(\si_0(y),\si_0(z))] - [{D}_2(\si_0(x),\si_0(y)), \si_0(z)]- [\si_0(y),{D}_2(\si_0(x),\si_0(z))]\\
=&\lambda_2(x,[y,z])-\lambda_2([x,y],z)-\lambda_2(y,[x,z])+l_0(x)(\lambda_2(y,z))-r_0(z)(\lambda_2(x,y))-l_0(y)(\lambda_2(x,z))
\end{align*}
Equating these two expressions, we obtain
\begin{align*}
&l_0(x)\lambda_2(y,z)-r_0(z)\lambda_2(x,y)-l_0(y)\lambda_2(x,z)-\lambda_1(l_3^\g(x,y,z)-\rho^L_{2}(x,y)(\lambda_0(z))\\
&-\rho_2^M(x,z)(\lambda_0(y))-\rho^R_{2}(y,z)(\lambda_0(x))+\lambda_2(x, [y,z])- \lambda_2([x,y],z) - \lambda_2(y,[x,z])=0.
\end{align*}
Hence $\lambda=(\lambda_0,\lambda_1,\lambda_2) \in {Z}^1(\fg,V)$. This completes the proof of the proposition.
\end{proof}
For each $(\beta, \alpha) \in D_\rho(\fg,V)$, we define
$$\Psi: D_\rho(\fg,V)\to \End (Z^2(\fg,V))$$
by
\begin{eqnarray*}
\Psi(\beta, \alpha)(\psi)(a)&=&\beta_0\psi(a)-\psi\alpha_1(a),\\
\Psi(\beta, \alpha)(\omega)(x,y)& = & \beta_0\omega(x,y)-\omega(\alpha_0(x),y)-\omega(x,\alpha_0(y)),\\
\Psi(\beta, \alpha)(\mu)(x,a)& = & \beta_1\mu(x,a)-\mu(\alpha_0(x),a)-\mu(x,\alpha_1(a)),\\
\Psi(\beta, \alpha)(\nu)(a,x)&=&\beta_1\nu(a,x)-\nu(\alpha_1(a),x)-\nu(a,\alpha_0(x)),\\
\Psi(\beta,\alpha)(\theta,\mu,\nu)(x,y,z)&=&\beta_0\theta(x,y,z)-\theta(\alpha_0(x),y,z)-\theta(x,\alpha_0(y),z)\\
&&-\theta(x,y,\alpha_0(z))-\mu(x,\alpha_2(y,z))+\nu(\alpha_2(x,y),z)\\
&&+\mu(y,\alpha_2(x,z)),
\end{eqnarray*}
\begin{cor}
Let $\mathcal{E}:0 \to V \to\hg \to \fg \to 0$ be an abelian extension of Lie algebras. Let $\rho: \fg \to  \End(V)$ denote the induced $\fg$-module structure on $V$. this defines a map
$$\varpi_{\mathcal{E}}: C_\rho(\fg,V) \to H^2(\fg,V)$$
given by
$$\varpi_{\mathcal{E}} (\beta, \alpha)= [\Psi(\beta, \alpha){(\psi,\omega,\mu,\nu,\theta)}].$$
\end{cor}

This $\varpi_{\mathcal{E}}$ is called the {\bf Wells map}.  We show that there is a left action of $ {D}_\rho(\fg,V)$ on $H^2(\fg,V)$ with respect to which $\varpi_{\mathcal{E}}$ is an inner derivation. Hence $\varpi_{\mathcal{E}}=0$ if and only if this action of $ {D}_\rho(\fg,V)$ on $H^2(\fg,V)$ is trivial.

\begin{prop}
Let $\mathcal{E}$ be an extension inducing $\rho$. Then $\varpi_{\mathcal{E}}$ is an inner derivation with respect to the action of $ {D}_\rho(\fg,V)$ on $H^2(\fg,V)$.
\end{prop}

\begin{proof}
Let $\mathcal{E}$ be an extension inducing $\rho$ and $\varpi_{\mathcal{E}}:  {D}_\rho(\fg,V) \to H^2(\fg,V)$ be the Wells map. Then for $(\beta, \alpha)$ in $C_\rho(\fg,V)$, we have

\begin{eqnarray*}
\psi_{(\beta, \alpha)}(a)&=&\beta_0\psi(a)-\psi\alpha_1(a),\\
\omega_{(\psi, \phi)}(x, y) & = & \beta_0\omega(x,y)-\omega(\alpha_0(x),y)-\omega(x,\alpha_0(y)),\\
\mu_{(\beta, \alpha)}(x,a)& = & \beta_1\mu(x,a)-\mu(\alpha_0(x),a)-\mu(x,\alpha_1(a)),\\
\nu_{(\beta, \alpha)}(a,x)&=&\beta_1\nu(a,x)-\nu(\alpha_1(a),x)-\nu(a,\alpha_0(x)),\\
(\theta,\mu,\nu)_{(\beta, \alpha)}(x,y,z)&=&\beta_0\theta(x,y,z)-\theta(\alpha_0(x),y,z)-\theta(x,\alpha_0(y),z)-\theta(x,y,\alpha_0(z))\\
&&-\mu(x,\alpha_2(y,z))+\nu(\alpha_2(x,y),z)+\mu(y,\alpha_2(x,z)),
\end{eqnarray*}
for all $x,y,z \in \g_0$, $a \in \g_1$.
Hence $\varpi_{\mathcal{E}}$ is an inner derivation with respect to the action of $ {D}_\rho(\fg,V)$ on $H^2(\fg,V)$.
\end{proof}

\begin{lem}\label{keylemma11}
Let  $(\beta, \alpha) \in  {D}_\rho(\fg,V)$. Then $(\beta, \alpha)$ is inducible  if and only if $(\beta, \alpha)$ is in the kernel of Wells map $\ker \varpi_{\mathcal{E}}$.
\end{lem}
\begin{proof}
Suppose $(\beta, \alpha)$ is inducible, then there exists a ${D}=({D}_0,{D}_1,{D}_2) \in \Der_V(\hg)$ such that $\Phi({D})=(\beta, \alpha)$. Let $\si=(\si_0,\si_1): \fg \to \hg$ be sections. Then any element of $\hg$ can be written uniquely as $\si_0(x)+v$, $\si_1(a)+m$ for some $v\in V_0$, $m\in V_1$ and $x \in \fg_0$, $a\in \g_1$. Now,
$$\alpha_0(x)=p_0 {D}_0 \si_0(x),~\alpha_1(a)=p_1 {D}_1 \si_1(a),~\alpha_2(x,y)=p_1 {D}_2 (\si_0(x),\si_0(y)).$$
 This implies
 $$p_0 \si_0 \alpha_0(x)=p_0 {D}_0 \si_0(x),~p_1 \si_1\alpha_1(a)=p_1 {D}_1 \si_1(a),~p_1\si_1\alpha_2(x,y)=p_1 {D}_2 (\si_0(x),\si_0(y)).$$
 And hence we obtain a linear map $\lambda \in \Hom(\fg,V)$ defined by
\begin{align*}
\lambda_0(x)=&~{D}_0 \sigma_0(x)-\sigma_0\alpha_0(x),\\
\lambda_1(a)=&~{D}_1 \sigma_1(a)-\sigma_1\alpha_1(a),\\
\lambda_2(x,y)=&~{D}_2(\sigma_0(x),\sigma_0(y))-\sigma_1\alpha_2(x,y),
\end{align*}

By the main Theorem of last section, we get
\begin{align*}
&\beta_0(\psi (a)) - \psi (\alpha_1(a)) = \dM_{V}(\lambda_1(a))-\lambda_0 (\dM_{\g}(a)),\\
\\
&\beta_0 (\omega (x, y)) - \omega (\alpha_0(x),y)-\omega(x, \alpha_0(y)) = l_0(x) (\lambda_0(y)) + r_0(y) (\lambda_0(x))  - \lambda_0 [x, y] +\dM_{V}\lambda_2(x,y)\\
&+\psi(\alpha_2(x,y)),\\
\\
&\beta_1 (\mu (x, a)) - \mu (\alpha_0(x), a) - \mu (x, \alpha_1(a))= l_0(x)( \lambda_1(a)) + r_1(a)( \lambda_0(x) ) - \lambda_1 [x, a]+\lambda_2(x,\dM_{\g}(a)),\\
\\
&\beta_1 (\nu (a, x)) - \nu (\alpha_1(a), x)- \nu (a, \alpha_0(x)) = r_0(x)( \lambda_1(a)) + l_1(a)( \lambda_0(x) ) - \lambda_1 [a, x]+\lambda_2(\dM_{\g}(a),x),\\
\\
&\beta_1\theta(x,y,z)-\theta(\alpha_0(x),y,z)-\theta(x,\alpha_0(y),z)-\theta(x,y,\alpha_0(z))
=-r_2(y,z)(\lambda_0(x))-m_2(x,z)(\lambda_0(y))\\
&-l_2(x,y)(\lambda_0(z))+\lambda_2(x,[y,z])-\lambda_2([x,y],z)-\lambda_2(y,[x,z])
+l_0(\alpha_0(x))(\lambda_2(y,z))+r_1(\alpha_2(y,z))(\lambda_0(x))\\
&-r_0(\alpha_0(z))(\lambda_2(x,y))
-l_1(\alpha_2(x,y))(\lambda_0(z))-l_0(\alpha_0(y))(\lambda_2(x,z))-r_1(\alpha_2(x,z))(\lambda_0(y)).
\end{align*}
Thus $\Psi(\beta, \alpha){(\psi,\omega,\mu,\nu,\theta)}=\partial^1 (\lambda)$ and
$$\varpi_{\mathcal{E}} (\beta, \alpha)= [\Psi(\beta, \alpha){(\psi,\omega,\mu,\nu,\theta)}]=0.$$
Then $(\beta,\alpha)$ is in the kernel of Wells map $\ker \varpi_{\mathcal{E}}$.
\end{proof}

\begin{lem}
The image of  $\Phi$ is equal to the kernel of  the Wells map $\varpi_{\mathcal{E}}$, i.e. $\operatorname{im}\Phi=\operatorname{ker} \varpi_{\mathcal{E}}$.
\end{lem}

\begin{proof}
One note that $ (\beta,\alpha) \in \operatorname{im} \Phi$ if and only if $ (\beta,\alpha)$ is inducible.
Then by the above Lemma \ref{keylemma11} we obtain the result.
\end{proof}

By Theorem \ref{main-thm1}, the pair $(\beta,\alpha) \in  {D}_\rho(\fg,V)$ is inducible if and only if $(\psi,\omega,\mu,\nu,\theta)_{(\beta,\alpha)}$ is a 2-coboundary. Thus $[(\psi,\omega,\mu,\nu,\theta)_{(\beta,\alpha)}] \in H^2(\fg,V)$ is an obstruction to inducibility of the pair $(\beta,\alpha) \in D_\rho(\fg,V)$. With the preceding discussion, we have derived the following exact sequence of groups associated to an abelian extension of Lie algebras and proving Theorem \ref{wells-sequence1}.

\begin{thm}\label{wells-sequence1}
Let $\mathcal{E}:0 \to V \to\hg \to \fg \to 0$ be an abelian extension of Lie algebras. Let $\rho: \fg \to  \End(V)$ denote the induced $\fg$-module structure on $V$. Then there is an exact sequence
\begin{equation}\label{wells-seq1}
0 \to {Z}^1(\fg,V) \stackrel{i}{\longrightarrow} \Der_V(\hg) \stackrel{\Phi}{\longrightarrow} {D}_\rho(\fg,V) \stackrel{\varpi_{\mathcal{E}}}{\longrightarrow} H^2(\fg,V).
\end{equation}
\end{thm}

\section{Special case: crossed modules}
In this section, we consider the inducibility of a pair of automorphisms in an abelian extension of crossed modules over Leibniz algebras.

We provide a special case of the representation and cohomology theory of Leibniz 2-algebras as well as some applications.
A Leibniz 2-algebra with $l_3=0$ is called strict Leibniz 2-algebra.
A strict Leibniz 2-algebra $\g=\g_{{1}}\oplus \g_{0}$ consists of the following data:
\begin{itemize}
\item[$\bullet$] a complex of vector spaces $l_1=\dM:\g_{{1}}{\longrightarrow}\g_0,$ 

\item[$\bullet$] bilinear maps $l_2=[\cdot,\cdot]:\g_i\times \g_j\longrightarrow
\g_{i+j}$, where $0\leq i+j\leq 1$,
\end{itemize}
 such that for any $x,y,z,t\in \g_0$ and $a,b\in \g_{{1}}$, the following conditions are satisfied:
\begin{itemize}
\item[$\rm(a)$] $\dM [x,a]=[x,\dM a],$ 
\item[$\rm(b)$] $\dM [a,x]=[\dM a,x],$ 
\item[$\rm(c)$] $[\dM a,b]=[a,\dM b],$  
\item[$\rm(d)$]     $[x,[y,z]]-[[x,y],z]-[y,[x,z]]=0,$
\item[$\rm(e)$]  $ [x,[y,a]]-[[x,y],a]-[y,[x,a]]=0,$
\item[$\rm(f)$]  $[x,[a,y]]-[[x,a],y]-[a,[x,y]]=0,$
\item[$\rm(g)$]  $[a,[x,y]]-[[a,x],y]-[x,[a,y]]=0.$
\end{itemize}

This lead to the concept of crossed modules over Leibniz algebras.

\begin{defi}
 A crossed module over Leibniz algebra $\pp$ is a  triple $(\pp_1,\pp_0, f)$  where $\pp_0$  is a Leibniz algebra,
$\pp_1$ is a representation of $\pp$ such that for any $x\in\pp_0$, $a, b\in\pp_1$, the map  ${f}:\pp_1\longrightarrow\pp_0$ satisfying the following conditions:
\begin{eqnarray}
 \label{crossed01} {f} (x\cdot {a})&=&[x,{f}({a})]_{\pp_0},\quad {f}({a}\cdot x)=[{f}({a}),x]_{\pp_0},\\
  \label{crossed02}~{f}({a})\cdot {b}&=&{a}\cdot {{f}({b})}.
\end{eqnarray}
The map ${f}:\pp_1\longrightarrow\pp_0$ is called an equivariant map if it satisfying \eqref{crossed01}.
\end{defi}

We remark that our definition is different from the concept of  crossed module of Leibniz algebras in \cite{LP,SL} where they require $\pp_1$ to be a Leibniz algebra, an action of  $\pp_0$ on $\pp_1$,  and a further condition
$$[{a},{b}]_{\pp_1}={f}({a})\cdot {b}={a}\cdot {f}({b}).$$
In fact, we find these conditions are not necessary for our strict Leibniz 2-algebras.

\begin{prop}\label{crossed:semidirectproduct}
  Given a representation of crossed module over Leibniz algebra $(\pp_1,\pp_0, f)$ on $(V_1,V_0,\varphi)$. Define on $(\pp_1\oplus V_1,\pp_0\oplus V_0,\widehat{f}=f+\varphi)$
  the following maps
\begin{equation}
\left\{\begin{array}{rcl}
\widehat{f}({a}+v)&\triangleq&{f}(a)+\varphi(v),\\
{[x+w, x'+w']}&\triangleq&[x, x']_\pp+l_0(x)(w')+r_0(x')(w),\\
{\widehat{l}(x+w)(a+v)}&\triangleq&l(x)(a)+l_0(x)(v)+r_1(a)(w),\\
{\widehat{r}(x+w)(a+v)}&\triangleq&r(x)(a)+r_0(x)(v)+l_1(a)(w),\\
\end{array}\right.
\end{equation}
for all $x, x'\in \pp_0$, ${a}\in \pp_1$, $v\in V_1$ and  $w,w'\in V_0$.
 Then $(\pp_1\oplus V_1,\pp_0\oplus V_0,\widehat{f})$ is a crossed module over Leibniz algebra, which is called the semidirect product of $(\pp_1,\pp_0, f)$ and $(V_1,V_0,\varphi)$.
\end{prop}

 \begin{defi}
Let $(\pp_1,\pp_0, f)$ and $(\pp_1',\pp_0', f')$ be crossed modules over Leibniz algebras. A crossed module  homomorphism $F=(F_0,F_1)$ from $\pp$ to $ \pp'$ consists of  linear maps $F_0:\pp_0\rightarrow \pp_0',~F_1:\pp_{{1}}\rightarrow \pp_{{1}}'$
such that the following equalities hold for all $ x,y\in \pp_{0},
a\in \pp_{{1}},$
\begin{itemize}
\item [$\rm(1)$] $F_0 f=f' F_1$,
\item[$\rm(2)$] $F_{0}[x,y]_\pp-[F_{0}(x),F_{0}(y)]'=0$,
\item[$\rm(3)$] $F_{1}[x,a]_\pp-[F_{0}(x),F_{1}(a)]'=0$,
\item[$\rm(4)$] $F_{1}[a,x]_\pp-[F_{1}(a),F_{0}(x)]'=0$.
\end{itemize}
\end{defi}

 \begin{defi}
Let $(\pp_1,\pp_0, f)$ be a crossed modules over Leibniz algebra. A  derivation of degree 0 of $\pp$ consists of $D=(D_0,D_1)$ from $\pp$ to $ \pp$ consists of  linear maps $D_0:\pp_0\rightarrow \pp_0,~D_1:\pp_{{1}}\rightarrow \pp_{{1}}$
such that the following equalities hold for all $ x,y\in \pp_{0},$
\begin{itemize}
\item [$\rm(1)$] $D_0 f=f' D_1$,
\item[$\rm(2)$] $D_{0}[x,y]_\pp-[D_{0}(x),(y)]-[(x),D_0(y)]=0,$
\item[$\rm(3)$] $D_{0}[x,a]_\pp-[D_{0}(x),(a)]-[(x),D_1(a)]=0$,
\item[$\rm(4)$] $D_{0}[a,x]_\pp-[D_{1}(a),(x)]-[(a),D_0(x)]=0$.
\end{itemize}
\end{defi}


Next, we study abelian extensions of crossed modules over Leibniz algebras.

\begin{defi}
Let $(\pp_1,\pp_0, f)$ be a crossed module over Leibniz algebra. An extension of $(\pp_1,\pp_0, f)$ is a
short exact sequence such that $\mathrm{Im}(i_0)=\mathrm{Ker}(p_0)$ and $\mathrm{Im}(i_1)=\mathrm{Ker}(p_1)$ in the following commutative diagram
\begin{equation}\label{eq:ext1}
\CD
  0 @>0>>  V_1 @>i_1>> \widehat{\pp_1} @>p_1>>  \pp_1  @>0>> 0 \\
  @V 0 VV @V \varphi VV @V \widehat{f} VV @V{f} VV @V0VV  \\
  0 @>0>> V_0 @>i_0>> \widehat{\pp_0} @>p_0>> \pp_0 @>0>>0
\endCD
\end{equation}
we call $(\widehat{\pp_1},\widehat{\pp_0}, \widehat f)$ an extension of $(\pp_1,\pp_0, f)$ by
$(V_1,V_0,\varphi)$, and denote it by $\widehat{\E}$.
It is called an abelian extension if $(V_1,V_0,\varphi)$ is an abelian crossed module over Leibniz algebra  (this means that the multiplication on $W$ is zero, the representation of $W$ on $V$ is trivial).
\end{defi}

A splitting $\sigma=(\sigma_0,\sigma_1):(\pp_1,\pp_0, f)\to (\widehat{\pp_1},\widehat \pp_0, \widehat f)$ consists of linear maps
$\sigma_0:{\pp_0}\to\widehat{\pp_0}$ and $\sigma_1:\pp_1\to \widehat{\pp_1}$
 such that  $p_0\circ\sigma_0=\id_{{\pp}}$, $p_1\circ\sigma_1=\id_{\pp_1}$ and $\widehat{f}\circ\sigma_1=\sigma_0\circ f$.

 Two extensions of Leibniz 2-algebras
 $\widehat{\E}:0\longrightarrow(V_1,V_0,\varphi)\stackrel{i}{\longrightarrow}(\widehat{\pp_1},\widehat \pp_0, \widehat f)\stackrel{p}{\longrightarrow}(\pp_1,\pp_0, f)\longrightarrow0$
 and  $\widetilde{\E}:0\longrightarrow(V_1,V_0,\varphi)\stackrel{j}{\longrightarrow}(\widetilde{\pp_1},\widetilde{\pp_0}, \widetilde{f})\stackrel{q}{\longrightarrow}(\pp_1,\pp_0, f)\longrightarrow0$ are equivalent,
 if there exists a homomorphism $F:(\widehat{\pp_1},\widehat \pp_1, \widehat f)\longrightarrow(\widetilde{\pp_1},\widetilde{\pp_0}, \widetilde{f})$  such that $F\circ i=j$, $q\circ
 F=p$.

Let $(\widehat{\pp_1},\widehat\pp_0, \widehat f)$ be an extension of $(\pp_1,\pp_0, f)$ by
$(V_1,V_0,\varphi)$ and $\sigma:(\pp_1,\pp_0, f)\to (\widehat{\pp_1},\widehat \pp_0, \widehat f)$ be a section.
Define the following maps:
\begin{equation}\label{eq:morphism}
\left\{\begin{array}{rlclcrcl}
l_0,r_0:&\pp&\longrightarrow& \End^{0}_{\dM_\frkh}(\frkh),&& l_0(x)(w+u)&\triangleq&[\sigma_0(x),w+u]_{\hp},\\
                                                  && & && r_0(x)(w+u)&\triangleq&[w+u, \sigma_0(x)]_{\hp},\\
l_1,r_1:&\hh&\longrightarrow& \End^{1}(\frkh),&&l_1(a)(w)&\triangleq&[\sigma_1(a),w]_{\hp},\\
                                                  &&& &&r_1(a)(w)&\triangleq&[u,\sigma_1(w)]_{\hp},\\
\end{array}\right.
\end{equation}
for all $x,y,z\in\pp$, $m\in \hh $.

\begin{prop}\label{pro:2-modules3}
With the above notations,  $(V_1,V_0,\varphi)$ is a representation of $(\pp_1,\pp_0,f)$.
Furthermore, this representation structure does not depend on the choice of the splitting $\sigma$.
Moreover,  equivalent abelian extensions give the same  representation on $(V_1,V_0,\varphi)$.
\end{prop}

Let $\sigma:(\pp_1,\pp_0, f)\to (\widehat{\pp_1},\widehat \pp_0, \widehat f)$  be a section of an abelian extension. Define the following linear maps:
$$
\begin{array}{rlclcrcl}
\theta:&\pp_1&\longrightarrow& V_0,&& \theta(a)&\triangleq&\widehat{f}\si(a)-\si(f (a)),\\
\omega:&\otimes^2\pp_0&\longrightarrow& V_0,&& \omega(x,y)&\triangleq&[\si(x),\si(y)]_{\widehat{\pp}}-\si[x,y]_\pp,\\
\mu:&\pp_0\otimes\pp_1&\longrightarrow&V_1,&& \mu(x,a)&\triangleq&[\si(x),\si(a)]_{\widehat{\pp}}-\si[x,a]_\pp,\\
\nu:&\pp_1\otimes\pp_0&\longrightarrow&V_1,&& \nu(a,x)&\triangleq&[\si(a),\si(x)]_{\widehat{\pp}}-\si[a,x]_\pp,
\end{array}
$$
for all $x,y,z\in\pp$, $m\in \hh $.

\begin{thm}\label{mainthm512}
There is a one-to-one correspondence between equivalence classes of abelian extensions and the second cohomology group ${H}^2((\pp_1,\pp_0,f),(V_1,V_0,\varphi))$.
\end{thm}

\subsection{Automorphisms of crossed modules}
Let $V$ and $\pp$ be crossed module over Leibniz algebras. Then an extension of $\pp$ by $V$ is a short exact sequence of crossed module over Leibniz algebra $$0 \to V \stackrel{i}{\to} \widehat{\pp} \stackrel{p}{\to} \pp \to 0,$$ where $\widehat{\pp} $ is a crossed module over Leibniz algebra. Without loss of generality, we may assume that $i$ is the inclusion map and we omit it from the notation.

Let $\Aut(V)$, $\Aut(\widehat{\pp})$ and $\Aut(\pp)$ denote the groups of all crossed module over Leibniz algebra automorphisms of $V$, $\widehat{\pp}$ and $\pp$, respectively.  Let $\Aut_V(\widehat{\pp})$ denote the group of all crossed module over Leibniz algebra automorphisms of $\widehat{\pp}$ which keep $V$ invariant as a set. Note that an automorphism ${F}=({F}_0, {F}_1) \in \Aut_V(\widehat{\pp})$ induces automorphisms ${F}|_V=(\beta_0, \beta_1) \in \Aut(V)$ given by
 $$\beta_0(v)={F}_0|_V(v),\quad \beta_1(m)={F}_1|_V(m),$$
 for all $u, v \in V_0, m \in V_1$.
 And we can also define a map $\overline{{F}}=(\alpha_0, \alpha_1) : \mathfrak{\pp} \rightarrow \mathfrak{\pp}$ by
\begin{align*}
\alpha_0(x)=p_0{F}_0\sigma_0(x),\quad \alpha_1(a)=p_1{F}_1\sigma_1(a),
\end{align*}
for all $x, y \in \pp_0, a \in \pp_1$.

\begin{thm}\label{main-thm31}
Let $0 \to V \to \widehat{\pp} \to \pp \to 0$ be an abelian extension of crossed module over Leibniz algebra and $(\beta, \alpha) \in \Aut(V)\times \Aut(\pp)$. Then the pair $(\beta, \alpha)$ is inducible if and only if exists a linear map $\lambda \in \mathrm{Hom}(\pp, V)$ satisfying the followings
\begin{align}
\label{CRO1}&\beta_0(\theta (a)) - \theta (\alpha_1(a)) = \psi(\lambda_1(a))-\lambda_0 (f(a)),\\
\label{CRO2}&\beta_0 (\omega (x, y)) - \omega (\alpha_0(x), \alpha_0(y)) =  l_0({\alpha_0(x)}) (\lambda_0(y)) + r_0({\alpha_0 (y)}) (\lambda_0(x))  - \lambda_0 ([x, y]), \\
\label{CRO3}&\beta_1 (\mu (x, a)) - \mu (\alpha_0(x), \alpha_1(a)) =  l_0({\alpha_ 0(x)})( \lambda_1(a)) + r_1({\alpha_1 (a)})( \lambda_0(x) ) - \lambda_1 ([x, a]),\\
\label{CRO4}&\beta_1 (\nu (a, x)) - \nu (\alpha_1(a), \alpha_0(x)) =  r_0({\alpha_0 (x)})( \lambda_1(a)) + l_1({\alpha_1 (a)})( \lambda_0(x) ) - \lambda_1 ([a, x]),\\
\label{CRO5}&\beta_0 { l_0}(x){\beta_0}^{-1} (v) ={ l_0}({\alpha_0 (x)})(v),\\
\label{CRO6}&\beta_0 { r_0}(x){\beta_0}^{-1} (v) = { r_0}({\alpha_0 (x)})  (v),\\
\label{CRO7}&\beta_1 l_0(x){\beta_1}^{-1}(m)= l_0(\alpha_0(x))(m),\\
\label{CRO8}&\beta_1 r_0(x){\beta_1}^{-1}(m)= r_0(\alpha_0(x))(m),\\
\label{CRO9}&\beta_1 r_1(a){\beta_0}^{-1}(v)= r_1(\alpha_1(a))(v),\\
\label{CRO10}&\beta_1 l_1(a){\beta_0}^{-1}(v)= l_1(\alpha_1(a))(v),
\end{align}

\end{thm}

\begin{rmk}\label{{rem-compatible}}
	A pair $(\be,\, \al) \in \Aut(V) \times \Aut(\pp)$ is called compatible if the conditions (\ref{CRO5}) -(\ref{CRO10}) hold. Equivalently, the following diagram commutes:
	$$
	\xymatrix{
		\pp \ar[d]_{\rho} \ar[r]^{\al} & \pp \ar[d]^{\rho}\\
		\End(V)\hspace*{3mm} \ar[r]_{f \mapsto F f F^{-1}} & \hspace*{3mm}\End(V)\\}
	$$
\end{rmk}
It is easy to see that the set
$$C_\rho(\pp,V)=\{(\beta,\, \alpha) \in \Aut(V) \times \Aut(\pp)|(\beta,\, \alpha) ~\text{satisfies ~ (\ref{CRO5}) -(\ref{CRO10}) } \}$$
of all compatible pairs is a subgroup of $\Aut(V) \times \Aut(\pp)$. Conditions (\ref{CRO5}) -(\ref{CRO10}) of the above theorem shows that every inducible pair is compatible.

Let
\begin{eqnarray*}
&&\ker \Phi=\Aut^{V,\pp}(\widehat{\pp})= \big\{{F} \in \Aut(\widehat{\pp})~|~ \Phi({F}_0)=(1_{V_0},1_{\pp_0}), ~\Phi({F}_1)=(1_{V_1},1_{\pp_1}),\big\}.
\end{eqnarray*}
and
\begin{eqnarray*}
&&{Z}^1(\pp,V) 
= \big\{\lambda \in {C}^1(\pp,V)~\big\}
\end{eqnarray*}
let $\lambda=(\lambda_0, \lambda_1)$ satisfies the following formulas
\begin{align*}
&\psi \lambda_1(a)-\lambda_0 f (a)=0,\\
& l_0(x)\lambda_0(y)+ r_0(y)\lambda_0(x)-\lambda_0[x,y]_\pp=0,\\
& l_0(x)\lambda_1(a)+ r_1(a)\lambda_0(x)-\lambda_1[x,a]_\pp=0,\\
& l_1(a)\lambda_1(x)+ r_0(x)\lambda_0(a)-\lambda_1[a,x]_\pp=0,
\end{align*}

For each $(\beta, \alpha) \in \C_\rho(\pp,V)$, we define
$$\Psi: C_\rho(\pp,V)\to \End (Z^2(\pp,V))$$
by
\begin{align*}
\Psi(\beta,\alpha)(\theta)=&~\theta_{(\beta, \alpha)}=\beta_0\theta(\alpha^{-1}_0),\\
\Psi(\beta, \alpha)(\omega)=&~\omega_{(\beta, \alpha)}=\beta_0\omega(\alpha^{-1}_0,\alpha^{-1}_0),\\
\Psi(\beta, \alpha)(\mu)=&~\mu_{(\beta, \alpha)}=\beta_1\mu(\alpha^{-1}_0,\alpha^{-1}_1),\\
\Psi(\beta, \alpha)(\nu)=&~\nu_{(\beta, \alpha)}=\beta_1\nu(\alpha^{-1}_1,\alpha^{-1}_0),
\end{align*}
It can be easily seen that $(\theta,\omega,\mu,\nu)_{(\theta, \phi)}$ is a 2-cocycle if $(\theta,\omega,\mu,\nu)$ is a 2-cocycle. Since $(\theta,\omega,\mu,\nu)$ is the 2-cocycle corresponding to the extension $\mathcal{E}$, this defines a map
$$\varpi_{\mathcal{E}}: C_\rho(\pp,V) \to H^2(\pp,V)$$
given by
$$\varpi_{\mathcal{E}} (\beta, \alpha)= [(\theta,\omega,\mu,\nu)_{(\beta, \alpha)}]-[(\theta,\omega,\mu,\nu)],$$
where $[(\theta,\omega,\mu,\nu,\theta)_{(\beta, \alpha)}]$ is the cohomology class of $(\theta,\omega,\mu,\nu,\theta)_{(\beta, \alpha)}$.
This $\varpi_{\mathcal{E}}$ is called the {\bf Wells map}.

\begin{thm}
Let $\mathcal{E}:0 \to V \to\hp \to \pp \to 0$ be an abelian extension of Leibniz 2-algebras. Let $\rho: \pp \to  \End(V)$ denote the induced $\pp$-module structure on $V$. Then there is an exact sequence
\begin{equation}
0 \to {Z}^1(\pp,V) \stackrel{i}{\longrightarrow} \Aut_V(\hp) \stackrel{\Phi}{\longrightarrow} C_\rho(\pp,V) \stackrel{\varpi_{\mathcal{E}}}{\longrightarrow} H^2(\pp,V).
\end{equation}
\end{thm}

\subsection{Derivations of crossed modules}

Let $\Der(V)$, $\Der(\widehat{\pp})$ and $\Der(\pp)$ denote the groups of all crossed module over Leibniz algebra derivations of $V$, $\widehat{\pp}$ and $\pp$, respectively.  Let $\Der_V(\widehat{\pp})$ denote the group of all crossed module over Leibniz algebra derivations of $\widehat{\pp}$ which keep $V$ invariant as a set. Note that a derivation ${D}=({D}_0, {D}_1) \in \Der_V(\widehat{\pp})$ induces derivations ${D}|_V=(\beta_0, \beta_1) \in \Der(V)$ given by
 $$\beta_0(v)={D}_0|_V(v),\quad \beta_1(m)={D}_1|_V(m),$$
 for all $u, v \in V_0, m \in V_1$.
 And we can also define a map $\overline{{D}}=(\alpha_0, \alpha_1) : \mathfrak{\pp} \rightarrow \mathfrak{\pp}$ by
\begin{align*}
\alpha_0(x)=p_0{D}_0\sigma_0(x),\quad \alpha_1(a)=p_1{D}_1\sigma_1(a),
\end{align*}
for all $x, y \in \pp_0, a \in \pp_1$.

\begin{thm}\label{main-thm32}
Let $0 \to V \to \widehat{\pp} \to \pp \to 0$ be an abelian extension of crossed module over Leibniz algebra and $(\beta, \alpha) \in \Der(V)\times \Der(\pp)$. Then the pair $(\beta,\alpha)$ is inducible if and only if exists a linear map $\lambda \in \mathrm{Hom}(\pp, V)$ satisfying the followings
\begin{align}
&\label{RCOCY1}\beta_0(\theta (a)) - \theta (\alpha_1(a)) = \dM_{V}(\lambda_1(a))-\lambda_0 (\dM_{\pp}(a)),\\
&\label{RCOCY2}\beta_0 (\omega (x, y)) - \omega (\alpha_0(x),y)-\omega(x, \alpha_0(y)) = l_0(x) (\lambda_0(y)) + r_0(y) (\lambda_0(x))  - \lambda_0 [x, y],\\
&\label{RCOCY3}\beta_1 (\mu (x, a)) - \mu (\alpha_0(x), a) - \mu (x, \alpha_1(a))= l_0(x)( \lambda_1(a)) + r_1(a)( \lambda_0(x) ) - \lambda_1 [x, a],\\
&\label{RCOCY4}\beta_1 (\nu (a, x)) - \nu (\alpha_1(a), x)- \nu (a, \alpha_0(x)) = r_0(x)( \lambda_1(a)) + l_1(a)( \lambda_0(x) ) - \lambda_1 [a, x],\\
&\label{RCOCY5}\beta_0l_0(x)(v)-l_0(\alpha_0(x))(v)
-l_0(x)(\beta_0(v))=0,\\
&\label{RCOCY6}\beta_0r_0(x)(v)
-r_0(\alpha_0(x))(v)-r_0(x)(\beta_0(v))=0,\\
&\label{RCOCY7}\beta_1l_0(x)(m)-l_0(\alpha_0(x))(m)
-l_0(x)(\beta_1(m))=0,\\
&\label{RCOCY8}\beta_1r_0(x)(m)
-r_0(\alpha_0(x))(m)-r_0(x)(\beta_1(m))=0,\\
&\label{RCOCY9}\beta_1l_1(a)(v)-l_1(\alpha_1(a))(v)
-l_1(a)(\beta_0(v))=0,\\
&\label{RCOCY10}\beta_1r_1(a)(v)
-r_1(\alpha_1(a))(v)-r_1(a)(\beta_0(v))=0,
\end{align}

\end{thm}

\begin{rmk}\label{{rem-compatible3}}
	A pair $(\be,\, \al) \in \Der(V) \times \Der(\pp)$ is called compatible if the conditions (\ref{RCOCY5}) -(\ref{RCOCY10}) hold. Equivalently, the following diagram commutes:
	$$
	\xymatrix{
		\pp \ar[d]_{\rho} \ar[r]^{\al} & \pp \ar[d]^{\rho}\\
		\End(V)\hspace*{3mm} \ar[r]_{f \mapsto D f D^{-1}} & \hspace*{3mm}\End(V)\\}
	$$
\end{rmk}
It is easy to see that the set
$$D_\rho(\pp,V)=\{(\beta,\, \alpha) \in \Der(V) \times \Der(\pp)|(\beta,\, \alpha) ~\text{satisfies ~ (\ref{RCOCY5}) -(\ref{RCOCY10}) } \}$$
of all compatible pairs is a subgroup of $\Der(V) \times \Der(\pp)$. Conditions (\ref{RCOCY5}) -(\ref{RCOCY10}) of the above theorem shows that every inducible pair is compatible.

Let
\begin{eqnarray*}
&&\ker \Phi=\Der^{V,\pp}(\widehat{\pp})= \big\{{D} \in \Der(\widehat{\pp})~|~ \Phi({D}_0)=(1_{V_0},1_{\pp_0}), ~\Phi({D}_1)=(1_{V_1},1_{\pp_1}),\big\}.
\end{eqnarray*}
and
\begin{eqnarray*}
&&{Z}^1(\pp,V) 
= \big\{\lambda \in {C}^1(\pp,V)~\big\}
\end{eqnarray*}

let $\lambda=(\lambda_0, \lambda_1)$ satisfies the following formulas
\begin{align*}
&\dM_\frkh \lambda_1(a)-\lambda_0 \dM_\g (a)=0,\\
&l_0(x)\lambda_0(y)+r_0(y)\lambda_0(x)-\lambda_0[x,y]_\g=0,\\
&l_0(x)\lambda_1(a)+r_1(a)\lambda_0(x)-\lambda_1[x,a]_\g=0,\\
&l_1(a)\lambda_1(x)+r_0(x)\lambda_0(a)-\lambda_1[a,x]_\g=0.
\end{align*}
For each $(\beta, \alpha) \in D_\rho(\pp,V)$, we define
$$\Psi: D_\rho(\pp,V)\to \End (Z^2(\pp,V))$$
by
\begin{align*}
\Psi(\beta, \alpha)(\theta)(a)=&\beta_0\theta(a)-\theta\alpha_1(a),\\
\Psi(\beta, \alpha)(\omega)(x,y)  = & \beta_0\omega(x,y)-\omega(\alpha_0(x),y)-\omega(x,\alpha_0(y)),\\
\Psi(\beta, \alpha)(\mu)(x,a) = & \beta_1\mu(x,a)-\mu(\alpha_0(x),a)-\mu(x,\alpha_1(a)),\\
\Psi(\beta, \alpha)(\nu)(a,x)=&\beta_1\nu(a,x)-\nu(\alpha_1(a),x)-\nu(a,\alpha_0(x)),\\
\end{align*}
 Since $(\theta,\omega,\mu,\nu)$ is the 2-cocycle corresponding to the extension $\mathcal{E}$, this defines a map
$$\varpi_{\mathcal{E}}: C_\rho(\pp,V) \to H^2(\pp,V)$$
given by
$$\varpi_{\mathcal{E}} (\beta, \alpha)= [(\theta,\omega,\mu,\nu)_{(\beta, \alpha)}]-[(\theta,\omega,\mu,\nu)],$$
where $[(\theta,\omega,\mu,\nu,\theta)_{(\beta, \alpha)}]$ is the cohomology class of $(\theta,\omega,\mu,\nu,\theta)_{(\beta, \alpha)}$.
This $\varpi_{\mathcal{E}}$ is the {Wells map}.

\begin{thm}
Let $\mathcal{E}:0 \to V \to\hp \to \pp \to 0$ be an abelian extension of Leibniz 2-algebras. Let $\rho: \pp \to  \End(V)$ denote the induced $\pp$-module structure on $V$. Then there is an exact sequence
\begin{equation}
0 \to {Z}^1(\pp,V) \stackrel{i}{\longrightarrow} \Aut_V(\hp) \stackrel{\Phi}{\longrightarrow} D_\rho(\pp,V) \stackrel{\varpi_{\mathcal{E}}}{\longrightarrow} H^2(\pp,V).
\end{equation}
\end{thm}

\section*{Acknowledgements}
This is a primary edition, something will be modified in the future.

\section*{Declarations}

 {\bf Competing interests} There are no conflicts of interest for this work.

\noindent {\bf Availability of data and materials} Data sharing is not applicable to this article as no datasets were generated or analyzed during the current study.

\vskip7pt
\footnotesize{
\noindent Wei Zhong\\
College of Mathematics and Information Science,\\
Henan Normal University, Xinxiang 453007, P. R. China};\\
 E-mail address: \texttt{{zhongweiHTU@yeah.net }}

 \vskip7pt
\footnotesize{
\noindent Tao Zhang\\
College of Mathematics and Information Science,\\
Henan Normal University, Xinxiang 453007, P. R. China;\\
 E-mail address: \texttt{{zhangtao@htu.edu.cn}}

\end{document}